\documentclass[11pt,reqno]{amsart}
\usepackage{amsmath,amsfonts,amsthm,color,amssymb, mathrsfs}

\usepackage{pxfonts}

\makeatletter
\@tfor\@tempa:=
\alpha \beta \gamma \delta \epsilon \zeta \eta \theta \iota \kappa \lambda \mu
\nu \xi \pi \rho \sigma \tau \upsilon \phi \chi \psi \omega \varepsilon
\vartheta \varpi \varrho \varsigma \varphi  
\do {
	\count255 \numexpr\@tempa+"7000\relax
	\expandafter\mathchardef\@tempa \count255 }
\makeatother

\usepackage{xcolor}
\usepackage[T1]{fontenc}
\usepackage{wasysym}
\usepackage[normalem]{ulem}
\usepackage{stmaryrd} 

\usepackage[left=1in, right=1in, top=1.1in,bottom=1.1in]{geometry}
\usepackage{enumitem}
\usepackage{hyperref}
\hypersetup{
	colorlinks   = true,
	citecolor    = blue,
	linkcolor=blue
}
\allowdisplaybreaks
\usepackage{csquotes}
\usepackage{graphicx}
\usepackage{pstricks}
\usepackage{lmodern}

\newtheorem{thm}{Theorem}[section]

\newtheorem{lem}[thm]{Lemma}

\newtheorem{defn}[thm]{Definition}

\def \a{{\alpha}}
\def \b{{\beta}}
\def \G{{\Gamma}}

\newcommand{\E}{\mathbb{E}}

\newcommand{\bean}{\begin{eqnarray*}}
	\newcommand{\eean}{\end{eqnarray*}}

\newcommand{\EE}{\mathbb{E}}

\newcounter{bean}
\newcommand{\benuma}{\setlength{\labelwidth}{.25in}
	
	\begin{list}
		{(\alph{bean})}{\usecounter{bean}}}
	\newcommand{\eenuma}{\end{list}}

\begin{document}
	
	\title[The Numerical OM functional for SDEs with fBM]{The Numerical Onsager-Machlup action functional for Euler discretized SDEs driven by fractional Brownian motion}
	\author[S. Liu]{Shanqi Liu}
	\address{ School of Sciences, Great Bay University, Dongguan, Guangdong, 523000, China}
	\address{School of Mathematics, University of Science and Technology of China, Hefei, Anhui, 230026, China}
	\email{shanqiliumath@126.com}
	
	\author[Q. Lian]{Qiqi Lian}
    \address{School of Mathematics, Sun Yat-sen University, Guangzhou, Guangdong, 510275, China}
    \address{School of Sciences, Great Bay University, Dongguan, Guangdong, 523000, China}
    \email{lianqq3@mail2.sysu.edu.cn}
	
	\author[J. Duan]{Jinqiao Duan}
	\address{Department of Mathematics and Department of Physics, Great Bay University, Dongguan, Guang
		dong, 523000, China}
	\email{duan@gbu.edu.cn}
	
	\author[H. Gao]{Hongjun Gao}
	\address{School of Mathematics, Southeast University, Nanjing, Jiangsu 211189, China}
	\email{hjgao@seu.edu.cn (Corresponding author)}

	\vspace{-2cm}
	\maketitle
	\vspace{-1cm}
	\begin{abstract}
		In this work, different from previous results, the explicit expression of numerical Onsager-Machlup action functional for Euler discretized SDEs driven by fractional Brownian motion is derived provided the drift coefficient and numerical reference path satisfy some suitable conditions. Then numerical fractional Euler-Lagrange equations for numerical Onsager-Machlup action functional are also obtained. Finally, numerical experiments to illustrate and support the theoretical findings. 
		
		\smallskip
		
	\end{abstract}
\section{Introduction}
Stochastic differential equations (SDEs), which serve as a fundamental mathematical framework for modeling dynamical systems subject to random perturbations, have been widely applications numerous fields \cite{Mao07,Ok13} such as financial engineering, systems biology, physical chemistry and control theory. However in many applications, random fluctuations exhibit temporal dependence and long-range memory, which cannot be captured by standard Brownian motion due to its independent increments. Fractional brownian motion (fBm) $B^H={B_t^H, t\in\mathbb{R}}$, characterized by the Hurst parameter $H\in(0,1)$ and its covariance
$$\E\big(B_t^HB_s^H\big)=\frac{1}{2}\big\{|t|^{2H}+|s|^{2H}-|t-s|^{2H}\big\},$$
provides a natural framework for modeling such phenomena. 

Consider the following stochastic differential equation
\begin{align}\label{main equation}
	d X_t=b(X_t)d t+d B^H_t, X_0=x,
\end{align}
where $b\in C^2_b(\mathbb{R})$. Moret and Nualart \cite{MN02} study the limit behaviour as $\varepsilon$ tends to zero of the ratios
\begin{align}\label{rotios}
	\gamma_{\varepsilon}(\phi)=\frac{\mathbb{P}(\|X-\phi\|\leq\varepsilon)}{\mathbb{P}(\|B^H\|\leq\varepsilon)},
\end{align}
where $\phi$ is a function such that $\phi-x$ belongs to the Cameron-Martin space associated with the fractional Brownian motion and $\|\cdot\|$ is a suitable norm. When the limit is of the form:
$$\lim\limits_{\varepsilon\rightarrow 0}\gamma_{\varepsilon}(\phi)=\exp (J(\phi)),$$
 the functional $J$ is called the Onsager-Machlup \cite{OM1,OM2} (OM) action functional associated with $(\ref{main equation})$ and with the norm $\|\cdot\|$.

The Onsager-Machlup action functional for classical stochastic differential equations driven by Brownian motion have been extensively studied over the past decades \cite{Cap95,Cap20,DB78,FK82,HT96,IW14,Str57,TM57}. 
Ikeda and Watanabe \cite{IW14} first derived the OM action functional for reference paths $\phi\in C^2([0,1],\mathbb{R}^d)$ under the supremum norm $\|\cdot\|$. 
Durr and Bach \cite{DB78} obtained the same result using Girsanov's transformation and quasi-translation invariance of the Wiener measure, together with the path integral representation based on a potential structure. 
Shepp and Zeitouni \cite{SZ92} showed that the result remains valid for any norm equivalent to the supremum norm, provided that $\phi-x$ belongs to the Cameron-Martin space. 
Capitaine \cite{Cap95} further extended these results to a large class of natural norms on Wiener space, including H\"{o}lder, Sobolev, and Besov norms. 
In \cite{HT96}, Hara and Takahashi computed the OM action functional for elliptic diffusion processes under the supremum norm. 
This result was subsequently generalized by Capitaine in \cite{Cap20} to norms dominating the $L^2$-norm on $\mathbb{R}^d$.

In \cite{MN02}, the authors proved that the OM action functional corresponding to \eqref{main equation} is given by
$$L(\phi,\dot{\phi})=-\frac{1}{2}\int_{0}^{1}\Big(\dot{\phi}_s-\big(K^H\big)^{-1}\int_{0}^{s}b(\phi_u)du\Big)^2ds-\frac{1}{2}\int_{0}^{1}b'(\phi_s)ds,$$
(the divergence term corrected from $d_H$ to $1$, following the
normalization discussion in \cite{Maayan17}), where the reference path $\dot{\phi}$ is a function such that $K^H \dot{\phi}=\phi-x$, $d_H$ is a constant depending on $H$ and the definition of $K^H$ can be found in subsection \ref{K^H}.

After deriving the Onsager-Machlup action functional, one can characterize the most probable transition path through the associated variational principle. However, in numerical simulations of most probable transition paths for stochastic differential equations driven by fractional Brownian motion, the displayed sample trajectories are typically generated by numerical discretization schemes rather than by the exact solution process itself. This naturally raises the following question: compared with the OM action functional associated with the original solution process, would it be more reasonable and practically meaningful to consider the OM functional corresponding to the numerical approximation?

On the other hand, increasing attention has recently been devoted to the dynamical behaviors of numerical methods, which provide effective tools for investigating intrinsic properties of the exact solution; see, for example, \cite{CDH,CHJS1,CHJS2,DK23}. In the context of stochastic dynamical systems driven by fractional Brownian motion, the lack of the Markov property and semimartingale structure introduces additional difficulties in both analysis and computation. It is therefore natural to ask whether numerical schemes preserve the Onsager-Machlup action functional structure of the original fBm-driven SDEs.

Motivated by these considerations, it is then natural to ask the following question: \textit{Can we derive the numerical Onsager-Machlup action functional for Euler-discretized SDEs driven by fractional Brownian motion?}

More precisely, we reformulate limiting behavior of ratios of the form
\begin{align}\label{numerical rotios}
	\gamma^n_{\varepsilon}(\phi^n)=\frac{\mathbb{P}(\|X^n-\phi^n\|\leq\varepsilon)}{\mathbb{P}(\|B^H\|\leq\varepsilon)},
\end{align}
where $\{\phi^n\}_{n\in\mathbb{N}}$ is a given family of reference paths, and $X^n$ denotes Euler approximation of \eqref{main equation} and is defined by
\begin{align}\label{Euler-main equation}
	X^n_t=x+\int_{0}^{t}b(X^n_{\eta_n(s)})d s+B^H_t,
\end{align}
or equivalently,
\begin{align}\label{Euler-main equation-1}
	X^n_t=X^n_{t_k}+b(X^n_{t_k})(t-t_k)+B^H_t-B^H_{t_k},
\end{align}
for $t_k\leq t<t_{t_{k+1}}$. 
For simplicity of presentation, we consider a uniform partition
of the interval $[0,1]$, given by $t_i=\frac{i}{n},i=0,\cdots,n$. 
Moreover, for every positive integer $n\in\mathbb{N}$, we
define $\eta_n(t)=t_k$ for $t_k\leq t <t_k+\frac{1}{n}$. 
For the Euler type schemes for SDEs driven by fractional Brownian motion, see e.g., \cite{HLN16,LT19,Neu06}.

We now state our main result on the numerical Onsager-Machlup action functional for Euler-discretized stochastic differential equations driven by fractional Brownian motion, considering both the singular and regular cases. Let $\{\phi^n\}_{n\in\mathbb N}$ be a family of numerical reference paths such that
$K^H\dot{\phi}^n=\phi^n-x.$
The corresponding numerical Onsager-Machlup action functional is given in both cases by
$$L^n(\phi^n,\dot{\phi}^n)=-\frac{1}{2}\int_{0}^{1}\Big(\dot{\phi}^n_s-\big(K^H\big)^{-1}\int_{0}^{s}b(\phi^n_{\eta_n(u)})du\Big)^2ds,$$
compared with the continuous Onsager-Machlup action functional, a notable feature of the discrete setting is the absence of a divergence correction term. 

The main technical difficulty in deriving the numerical Onsager-Machlup action functional for fractional Brownian motion lies in the structure of the discretized Volterra kernel. The expressions $\mathcal{J}^n_{12}$ and $\mathcal{I}^n_{12}$(in \S 4) involve highly coupled double summation terms over the partition intervals  after time discretization, the estimates must be carefully reorganized at the discrete level in order to control both the singular (regular) behavior of the fractional kernel and the coupling between different partition intervals.

To handle these difficulties, we introduce suitable decompositions of the discretized fractional operators and prove the nuclearity of the corresponding integral operators. This allows us to identify the numerical Onsager-Machlup action functional and to clarify how the discretized Volterra structure affects its form.

Based on these results, we derive the numerical Euler-Lagrange equations associated with the numerical Onsager-Machlup action functional. These equations describe the most probable paths of the discretized solution process \eqref{Euler-main equation}. We also present numerical experiments that support the theoretical results.

The rest of the paper is organized as follows. In Section \ref{sec:opmr}, we recall recall some classical results for fractional calculus and fractional Brownian motion. And we set our setting in Section \ref{framework}. The deviations of the numerical Onsager-Machlup action functional for Euler discretized SDEs are presented in Section \ref{DEVIATION}.  In Section \ref{TMPP}, we further derive Euler-Lagrange fractional equations of the numerical OM action functional in singular case and regular case, some numerical examples are also provided. Finally, we postpone some notations, definitions and lemmas to the end in Appendix \ref{Appendix}. Throughout the paper C denotes a
positive constant whose value may change from line to line. The dependence of
constants on parameters when relevant will be denoted by special
symbols or by mentioning the parameters in brackets, e.g.
$C_{\a,
\b}$.

\section{Preliminaries}\label{sec:opmr}
In this section, we recall some classical results for fractional calculus and fractional Brownian motion.
\subsection{Fractional calculus}
For $f\in L^1[a,b]$ and $\a>0$ the right-side fractional Riemann-Liouville integrals of $f$ of order $\a$ on $(a,b)$ are defined at all $x$ by
$$(I^{\a}_{a^+}f)(x)=\frac{1}{\G(\a)}\int_{a}^{x}(x-y)^{\a-1}f(y)dy,$$
where $\G$ denotes the Euler function.

This integral extends the usual $n$-order iterated integrals of $f$ for $\a=n\in \mathbb{N}$.
We have the first composition formula
$$I^{\a}_{a^+}(I^{\beta}_{a+}f)=I^{\a+\beta}_{a^+}f.$$

When $\a p>1$ any function in $I^{\a}_{a^+}(L^p)$ is $\big(\a-\frac{1}{p}\big)$-H$\mathrm{\ddot{o}}$lder continuous.  The fractional derivative can be introduced as an inverse operation. If $1\leq p<\infty$, we denote by $I^{\a}_{a^+}(L^p)$ the image of $L^p([a,b])$ by the operator $I^{\a}_{a^+}$. If $f\in I^{\a}_{a^+}(L^p)$, the function $\phi$ such that $f=I^{\a}_{a^+}\phi$ is unique in $L^p$ and it agrees with the left-sided Riemann-Liouville derivative of $f$ of order $\a$ defined by
$$(D^{\a}_{a^{+}}f)(x)=\frac{1}{\G(1-\a)}\frac{d}{dx}\int_{a}^{x}\frac{f(y)}{(x-y)^{\a}}dy.$$
On the other hand, any H$\mathrm{\ddot{o}}$lder continuous function of order $\b>\a$ has a fractional derivative of order $\a$. The derivative of $f$ has the following Weyl representation:
\begin{align}\label{weyl representation}
	(D^{\a}_{a^+}f)(x)&=\frac{1}{\G(1-\a)}\Big(\frac{f(x)}{(x-a)^{\a}}+\a\int_{a}^{x}\frac{f(x)-f(y)}{(x-y)^{\a+1}}dy\Big)\mathbb{I}_{(a,b)}(x),
\end{align}
where the convergence of the integrals at the singularity $x=y$ holds in $L^p$-sense.

Recall that by construction for $f\in I^{\a}_{a^+}(L^p)$,
$$I^{\a}_{a^+}(D^{\a}_{a^+}f)=f$$
and for general $f\in L^1([a,b])$ we have
$$D^{\a}_{a^+}(I^{\a}_{a^+}f)=f.$$
If $f\in I^{\a+\b}_{a^+}(L^1), \a\geq0,\b\geq0,\a+\b\leq1$, we have the second composition formula
$$D^{\a}_{a^+}(D^{\beta}_{a+}f)=D^{\a+\beta}_{a^+}f.$$
The following estimate for the norm of the fractional integral will be used later in this paper,
\begin{align}\label{fractional integral estimate}
	\|I^{\a}_{a^+}f\|_{L^p([a,b])}\leq\frac{(b-a)^{\a}}{\a|\G(\a)|}\|f\|_{L^p([a,b])},
\end{align}
provided $f\in L^p([a,b])$.
\subsection{The Wiener integral with respect to fractional Brownian motion}\label{K^H}
The fractional Brownian motion has the integral representation in law:
\begin{align}\label{relation between B and W}
	B_t^H=\int_{0}^{t}K^H(t,s)dW_s,
\end{align}
where $W$ is a standard Brownian motion and $K^H$ is the square integral kernel:
\begin{align}\label{K^H(r,u)}
	K^H(r,u)=c_H(r-u)^{H-\frac{1}{2}}+c_H(\frac{1}{2}-H)\int_{u}^{r}(\theta-u)^{H-\frac{3}{2}}\Big(1-(\frac{u}{\theta})^{\frac{1}{2}-H}\Big)d\theta,
\end{align}
with
$$c_H=\Big(\frac{2H\G(\frac{3}{2}-H)}{\G(H+\frac{1}{2})\G(2-2H)}\Big)^{\frac{1}{2}}.$$
We also denote by $K^H$ the operator in $L^2([0,1])$ associated with the kernel $K^H$, that is
$$\big(K^Hh\big)(s)=\int_{0}^{s}K^H(s,r)h(r)dr.$$
If $p>\frac{1}{H+\frac{1}{2}}$ the operator $K^H$ is continuous in $L^p$, and we denote by
$$\mathcal{H}^p=\{K^H h, h\in L^p([0,1])\}$$
the image of $L^p([0,1])$ by $K^H$.

It is proved in \cite{DU99} (Theorem $2.1$) that the operator $K^H$ can be expressed in terms of fractional integrals as follows,
$$\Big(K^Hh\Big)(s)=d_HI^{1-2\alpha}_{0+}s^{\alpha}I^{\alpha}_{0+}s^{-\alpha}h
, \quad \text{if}\ H\leq \frac{1}{2},$$
$$\Big(K^Hh\Big)(s)=d_HI^{1}_{0+}s^{\alpha}I^{\alpha}_{0+}s^{-\alpha}h
, \quad \text{if}\ H\geq \frac{1}{2},$$
where $d_H=c_H\Gamma(H+1/2), \a=|H-\frac{1}{2}|$ and $h\in H$. Hence, the inverse operator is defined by
$$\Big(\big(K^H\big)^{-1}h\Big)(s)=d_H^{-1}s^{\alpha}D^{\alpha}_{0+}s^{-\alpha}h'\quad \text{if}\ H\geq \frac{1}{2},$$
$$\Big(\big(K^H\big)^{-1}h\Big)(s)=d_H^{-1}s^{\alpha}D^{\alpha}_{0+}s^{-\alpha}D^{1-2\alpha}_{0+}h\quad \text{if}\ H\leq \frac{1}{2},$$
for all $h\in I^{H+\frac{1}{2}}_{0+}(L^2([0,1]))$. If $h$ is differentiable it can be proved in the last case
that
$$\Big(\big(K^H\big)^{-1}h\Big)(s)=d_H^{-1}s^{-\alpha}I^{\alpha}_{0+}s^{\alpha}h'.$$

\section{Setting}\label{framework}
In this section, we first introduce the structure of numerical reference path. Then we further convert the problem of deriving numerical Onsager-Machlup action functional into clearer conditional exponential moments by Girsanov's theorem.
\subsection{The structure of numerical reference path}
Let $B^H={B_t^H, t\in[0,1]}$ be a fractional Brownian motion with Hurst index $0<H<1$ ($H\neq\frac{1}{2}$) defined on the given complete filtered probability space $(\Omega,\mathcal{F},(\mathcal{F}_t)_{t\ge0},\mathbb{P})$. Consider the following discretized stochastic differential equation:
\begin{equation}\label{nsde}
	X^n_t=x+\int_{0}^{t}b(X^n_{\eta_n(s)})d s+B^H_t,
\end{equation}
where $b$ is a function in $C^2_b([0,1])$.

The operator $K^H$ associated with the kernel \eqref{K^H(r,u)} defines an isomorphism from $L^2([0,1])$ to $I_{0^{+}}^{H+1/2}(L^2([0,1]))$. Hence, the space $\mathcal{H}^2=\{K^H h, h\in L^2([0,1])\}$ is included in the sapce of H$\mathrm{\ddot{o}}$lder continuous functions of order $H$. Let $\phi^n_{n\in\mathrm{N}}$ be a family numerical reference path such that $\phi^n-x\in\mathcal{H}^p$ with $p>\frac{1}{H}$ for every $n\in\mathrm{N}$. We will denote $\dot{\phi}^n$ the function in $L^p([0,1])$ such that
\begin{align}\label{structure of numerical reference path}
	K^H\dot{\phi}^n=\phi^n-x.
\end{align}
Our purpose is to study the limit behavior of the rotios for every $n\in\mathbb{N}$
$$\gamma^n_{\varepsilon}(\phi)=\frac{\mathbb{P}(\|X^n-\phi^n\|\leq\varepsilon)}{\mathbb{P}(\|B^H\|\leq\varepsilon)},$$
when $\varepsilon$ tends to zero and $\|\cdot\|$ is a suitable norm.
\subsection{Girsanov transform for discretized stochastic differential equation}
Consider the following auxiliary SDE on $\mathbb{R}$:
\begin{equation}
	\tilde{X}^n_t=x+{\phi}^{n}_t+B^H_t.\nonumber
\end{equation}
Define
\begin{align}
	\tilde{W}_t=W_t-\int_{0}^{t} \eta^n_s ds,\nonumber
\end{align}
where
\begin{align}
	\eta^n_s=\Big((K^H)^{-1}\Big(b(\tilde{X}^n_{\eta_n(u)})\Big)(s)-\dot{\phi}^{n}_s\Big), \quad s \in[0, 1],\nonumber
\end{align}
and set
\begin{align}
	\tilde{B}_t^H&=\int_{0}^{t}K^H(t,s)d\tilde{W}_s.
\end{align}
Here, $W$ denotes the Brownian motion in \eqref{relation between B and W}.

Next, introduce the probability measure $\tilde{\mathbb P}$ by
\begin{equation}\label{eta}
	\frac{d\tilde{\mathbb P}}{d\mathbb P}=
	\exp\left(
	\int_0^1\eta_s^ndW_s
	-\frac12\int_0^1|\eta_s^n|^2ds
	\right).
\end{equation}
In Lemma \ref{girsanov}, we shall verify that $\eta^n$ is adapted and satisfies Novikov's condition. Consequently, the exponential process in \eqref{eta} defines a martingale, and Girsanov's theorem implies that $\tilde W$ is a standard Brownian motion under $\tilde{\mathbb P}$.

By the Volterra representation of fractional Brownian motion, it follows that $\tilde B^H$ is again a fractional Brownian motion on the probability space $(\Omega,\mathcal F,\tilde{\mathbb P})$. Therefore, under $\tilde{\mathbb P}$, the process $\tilde X^n$ satisfies
\begin{equation}
	\tilde{X}^n_t=x+\int_{0}^{t}b(\tilde{X}^n_{\eta_n(s)})d s+\tilde{B}_t^H.\nonumber
\end{equation}
This change-of-measure argument enables us to reformulate the small-ball probability problem for the original discretized equation as an equivalent problem under the transformed probability measure $\tilde{\mathbb P}$, where the centered process $\tilde X^n-\phi^n$ becomes a fractional Brownian motion.
\begin{align}\label{all}
	&\mathbb{P}(\|X^n-\phi^n\|\leq\varepsilon)\nonumber\\
	&=\tilde{\mathbb{P}}(\|Y^n-\phi^{n}\|\leq \varepsilon)=\E\Big(\frac{d\tilde{\mathbb{P}}}{d\mathbb{P}}\mathbb{I}_{\|B^H\|\leq\varepsilon})\nonumber\\
	&=\E\Big(\!\exp\!\Big(\!\int_{0}^{1}\eta^n_{s}ds-\frac{1}{2}\int_{0}^{1}|\eta^n_{s}|^2ds\Big)\mathbb{I}_{\|B^H\|\leq\varepsilon}\Big)\nonumber\\
	&=\E\Big(\!\exp\Big(\!\int_{0}^{1}\!\!\Big((K^H)^{-1}\!
	\Big(\!b(\tilde{Y}^n_{\eta_n(u)}\!\Big)(s)-\dot{\phi}^{n}_s\Big)ds
	-\frac{1}{2}\int_{0}^{1}\!\!\Big((K^H)^{-1}\Big(b(\tilde{Y}^n_{\eta_n(u)})\Big)(s)-\dot{\phi}^{n}_s\Big)^2\!ds\!\Big)
	\mathbb{I}_{\|B^H\|\leq\varepsilon}\!\Big)\nonumber\\
	&=\E\Big(\!\exp\Big(\!\int_{0}^{1}\!\big((K^H)^{-1}b(\phi^{n}_{\eta_n(u)}+B_{\eta_n(u)}^H)\big)(s)dW_s+\int_{0}^{1}(-\dot{\phi}^{n}_s)dW_s\nonumber\\
	&\quad+\frac{1}{2}\int_{0}^{1}\big|\dot{\phi}^{n}_s-\big((K^H)^{-1}b(\phi^n_{\eta_n(u)})\big)(s)\big|^2ds
	-\frac{1}{2}\int_{0}^{1}\big|\dot{\phi}^{n}_s-\big((K^H)^{-1}b(\phi^n_{\eta_n(u)})\big)(s)\big|^2ds\nonumber\\
	&\quad
	-\frac{1}{2}\int_{0}^{1}\Big((K^H)^{-1}\Big(b(\tilde{Y}^n_{\eta_n(u)})\Big)(s)-\dot{\phi}^{n}_s\Big)^2ds\Big)\mathbb{I}_{\|B^H\|\leq\varepsilon}\Big)\nonumber\\
	&=\E\Big(\exp(\mathcal{A}^n_1+\mathcal{A}^n_2+\mathcal{A}^n_3+\mathcal{A}^n_4)\mathbb{I}_{\|B^H\|\leq\varepsilon}\Big)\times\exp\Big(-\frac{1}{2}\int_{0}^{1}\big|\dot{\phi}^{n}_s-\big((K^H)^{-1}b(\phi^n_{\eta_n(u)})\big)(s)\big|^2ds\Big),
\end{align}
where
\begin{align}
	\mathcal{A}^n_1&=\int_{0}^{1}\big((K^H)^{-1}b(\phi^{n}_{\eta_n(u)}+B_{\eta_n(u)}^H)\big)(s)dW_s,\nonumber\\
	\mathcal{A}^n_2&=\int_{0}^{1}-\dot{\phi}^{n}_sdW_s,\nonumber\\
	\mathcal{A}^n_3&:=\int_{0}^{1}\dot{\phi}^{n}_s\cdot\Big((K^H)^{-1}\big(b(\phi^{n}_{\eta_n(u)}+B_{\eta_n(u)}^H)-b(\phi^n_{\eta_n(u)})\big)(s)\Big)ds,\nonumber\\
	\mathcal{A}^n_4&:=\frac{1}{2}\int_{0}^{1}\Bigg(\Big(\big((K^H)^{-1}b(\phi^n_{\eta_n(u)})\big)(s)\Big)^2-\Big(\big((K^H)^{-1}b(\phi^{n}_{\eta_n(u)}+B_{\eta_n(u)}^H)\big)(s)\Big)^2\Bigg)ds.\nonumber
\end{align}
\section{The numerical Onsager-Machlup action functional for discretized SDEs}\label{DEVIATION}
In this section, we derive the numerical Onsager-Machlup action functional associated with \eqref{nsde} in both the singular regime $H<\frac12$ and the regular regime $H>\frac12$. Recalling the ratios introduced in \eqref{rotios}, we first present the numerical Onsager-Machlup action functional corresponding to the singular case $H<\frac12$.

\begin{thm}
	\label{th:singular case}
	Let $X^n$ be the solution of discretized SDE \eqref{Euler-main equation} 
	with Hurst index $\frac{1}{4}<H<\frac{1}{2}$. Let $\{\phi^n\}_{n\in\mathbb{N}}$ be a family of numerical reference paths such that $\phi^n-x\in\mathcal{H}^p$ with $p>\frac{1}{H}$ and assume $b\in C^2_b(\mathbb{R})$. 
	Then the numerical Onsager-Machlup action functional of $X^n$ for the norms $\|\cdot\|_{\b}$  with $0<\beta<H-\frac{1}{4}$ 
	and $\|\cdot\|_{\infty}$ exists and is given by
	$$L^n(\phi^n,\dot{\phi}^n)=-\frac{1}{2}\int_{0}^{1}\big|\dot{\phi}^{n}_s-d_H^{-1}s^{-\a}\big(I^{\a}_{0+}u^{\a}b(\phi^n_{\eta_n(u)})\big)(s)\big|^2ds,$$
	where $\a=\frac{1}{2}-H$ and numerical reference path $\phi^n, n\in\mathbb{N}$ is the function such that $K^H\dot{\phi}^n=\phi^n-x$.
\end{thm}
\begin{proof}
	We will prove the theorem for H$\mathrm{\ddot{o}}$lder norm. The proof is the same for the supremum norm. Recall operator $(K^H)^{-1}$ is defined by
	$$\Big(\big(K^H\big)^{-1}h\Big)(s)=s^{-\alpha}\big(I^{\alpha}_{0+}u^{\alpha}h'\big)(s).$$
	where $\a=\frac{1}{2}-H$ when $H<\frac{1}{2}.$ So we can rewrite small probability \eqref{all}.
	\begin{align}\label{small all}
		&\mathbb{P}(\|X-\phi^{n}\|_{\beta}\leq\varepsilon)\nonumber\\&=\E\Big(\exp(\mathcal{J}^n_1+\mathcal{J}^n_2+\mathcal{J}^n_3+\mathcal{J}^n_4)\mathbb{I}_{\|B^H\|\leq\varepsilon}\Big)\times\exp\Big(-\frac{1}{2}\int_{0}^{1}\big|\dot{\phi}^{n}_s-s^{-\a}\big(I^{\a}_{0+}u^{\a}b(\phi^n_{\eta_n(u)})\big)(s)\big|^2ds\Big),
	\end{align}
	where
	\begin{align}
		\mathcal{J}^n_1&=\int_{0}^{1}d_H^{-1}s^{-\alpha}\big(I^{\alpha}_{0+}u^{\alpha}b(\phi^{n}_{\eta_n(u)}+B_{\eta_n(u)}^H)\big)(s)dW_s,\nonumber\\
		\mathcal{J}^n_2&=\int_{0}^{1}-\dot{\phi}^{n}_sdW_s,\nonumber\\
		\mathcal{J}^n_3&=\int_{0}^{1}d_H^{-1}\dot{\phi}^{n}_s\cdot\Big(s^{-\alpha}\big(I^{\alpha}_{0+}u^{\alpha}\big(b(\phi^{n}_{\eta_n(u)}+B_{\eta_n(u)}^H)-b(\phi^n_{\eta_n(u)})\big)\big)(s)\Big)ds,\nonumber\\\mathcal{J}^n_4&=\frac{1}{2}\int_{0}^{1}\Bigg(\Big(d_H^{-1}s^{-\alpha}\big(I^{\alpha}_{0+}u^{\alpha}b(\phi^n_{\eta_n(u)})\big)(s)\Big)^2-\Big(d_H^{-1}s^{-\a}\big(I^{\alpha}_{0+}u^{\alpha}b(\phi^{n}_{\eta_n(u)}+B_{\eta_n(u)}^H)\big)(s)\Big)^2\Bigg)ds,\nonumber
	\end{align}
	then, by Lemma \ref{Separation lemma}, we could deal with each term independently.
	
	$\clubsuit \ \text{Term}$ \ $\mathcal{J}^n_2$
	
	Applying Theorem \ref{no random function} to $f=-c\dot{\phi}^{n}_s$ and Lemma \ref{norm}, we have
	\begin{align}\label{I_2}
		\limsup _{\varepsilon \rightarrow 0}\ \E(\exp(c\mathcal{J}^n_2)|\|B^H\|_{\beta}<\varepsilon)\leq1,
	\end{align}
	for every real number $c$.
	
	$\clubsuit  \ \text{Term}$ \ $\mathcal{J}^n_3$
	
	So under the condition $\|B^H\|_{\beta}<\varepsilon$, by using the fact $b$ is Lipschitz continuous and bounded with constant $L$, we have that (where $v=\frac{u}{s}$)
	\begin{align}\label{ready1}
		&\Big|s^{-\alpha}\big(I^{\alpha}_{0+}u^{\alpha}\big(b(\phi^{n}_{\eta_n(u)}+B_{\eta_n(u)}^H)-b(\phi^n_{\eta_n(u)})\big)\big)(s)\Big|\nonumber\\&=\frac{1}{\G(\a)}s^{-\alpha}\Bigg|\int_{0}^{s}u^{\alpha}(s-u)^{\a-1}\Big(b(\phi^{n}_{\eta_n(u)}+B_{\eta_n(u)}^H)-b(\phi^n_{\eta_n(u)})\Big)du\Bigg|\nonumber\\&\leq\frac{C_{L}}{\G(\a)} s^{-\alpha}\int_{0}^{s}u^{\alpha}(s-u)^{\a-1}|\eta_n(u)|^{\b}\|B^H\|_{\b}du\nonumber\\&\leq \frac{C_{L}}{\G(\a)} \varepsilon s^{-\alpha}\int_{0}^{s}u^{\alpha}(s-u)^{\a-1}du\nonumber\\&=\frac{C_{L}}{\G(\a)}\varepsilon s^{\alpha}\int_{0}^{1}v^{\alpha}(1-v)^{\a-1}dv=C_{L}\frac{\b(1+\a,\a)}{\G(\a)} s^{\alpha}\varepsilon.
	\end{align}
	We now deal with the term $\mathcal{J}^n_3$.
	\begin{align}
		|\mathcal{J}^n_3|&=\int_{0}^{1}d_H^{-1}\dot{\phi}^{n}_s\cdot\Big(s^{-\alpha}\big(I^{\alpha}_{0+}u^{\alpha}\big(b(\phi^{n}_{\eta_n(u)}+B_{\eta_n(u)}^H)-b(\phi^n_{\eta_n(u)})\big)\big)(s)\Big)ds\nonumber\\&\leq C_{L,H}\frac{\b(1+\a,\a)}{\G(\a)}\varepsilon \int_{0}^{1}s^{\a}|\dot{\phi}^{n}_s|ds\nonumber\\&\leq C_{\a,L}\varepsilon.\nonumber
	\end{align}
	Hence
	\begin{align}\label{I_3}
		\limsup _{\varepsilon \rightarrow 0}\ \E(\exp(c\mathcal{J}^n_3)|\|B^H\|<\varepsilon)\leq1,
	\end{align}
	for every real number $c$.
	
	$\clubsuit \ \text{Term} $\ $\mathcal{J}^n_4$
	
	For the term $\mathcal{J}^n_4$, we have
	\begin{align}
		|\mathcal{J}^n_4|&\leq \frac{1}{2}\int_{0}^{1}\Bigg|\Big(d_H^{-1}s^{-\alpha}\big(I^{\alpha}_{0+}u^{\alpha}b(\phi^n_{\eta_n(u)})\big)(s)\Big)^2-\Big(d_H^{-1}s^{-\a}\big(I^{\alpha}_{0+}u^{\alpha}b(\phi^{n}_{\eta_n(u)}+B_{\eta_n(u)}^H)\big)(s)\Big)^2\Bigg|ds\nonumber\\&\leq \frac{1}{2}\int_{0}^{1}\Bigg(d_H^{-1}s^{-\alpha}\Big(I^{\alpha}_{0+}u^{\alpha}\big(b(\phi^{(n)}_{\eta_n(u)}+B_{\eta_n(u)}^H)-b(\phi^n_{\eta_n(u)})\big)\Big)(s)\Bigg)^2ds\nonumber\\&\quad+\int_{0}^{1}\Bigg|d_H^{-2}s^{-\a}\Big(I^{\a}_{0+}u^{\a}\big(b(\phi^{(n)}_{\eta_n(u)}+B_{\eta_n(u)}^H)-b(\phi^n_{\eta_n(u)})\big)\Big)(s)\cdot s^{-\a}\big(I^{\a}_{0+}u^{\a}b(\phi^n_{\eta_n(u)})\big)(s)\Bigg|ds\nonumber\\\nonumber&:=\mathcal{J}^n_{41}+\mathcal{J}^n_{42}.
	\end{align}
	Using \eqref{ready1} we obtain
	\begin{align}
		|\mathcal{J}^n_{41}|=&\frac{1}{2}\int_{0}^{1}\Bigg(d_H^{-1}s^{-\alpha}\Big(I^{\alpha}_{0+}u^{\alpha}\big(b(\phi^{(n)}_{\eta_n(u)}+B_{\eta_n(u)}^H)-b(\phi^n_{\eta_n(u)})\big)\Big)(s)\Bigg)^2ds\nonumber\\\leq& C_{L,H}\frac{\b^2(1+\a,\a)}{2(2\a+1)\G(\a)^2}\varepsilon^2,\nonumber
	\end{align}
	and from \eqref{fractional integral estimate}, \eqref{ready1} and $b$ is bounded, we have ($p=1,f=s^{\a}b(\phi^n_{\eta_n(u)})$)
	\begin{align}
		|\mathcal{J}^n_{42}|=&\int_{0}^{1}\Bigg|d_H^{-2}s^{-\a}\Big(I^{\a}_{0+}u^{\a}\big(b(\phi^{(n)}_{\eta_n(u)}+B_{\eta_n(u)}^H)-b(\phi^n_{\eta_n(u)})\big)\Big)(s)\cdot s^{-\a}\big(I^{\a}_{0+}u^{\a}b(\phi^n_{\eta_n(u)})\big)(s)\Bigg|ds\nonumber\\\leq&C_{L}\frac{\beta(1+\a,\a)}{\G(\a)}\varepsilon\int_{0}^{1}\big|I^{\a}_{0+}u^{\a}b(\phi^n_{\eta_n(u)})\big|(s)ds\nonumber\\\leq&C_{L,H}\frac{\beta(1+\a,\a)}{\a\G(\a)^2}\varepsilon\int_{0}^{1}s^{\a}|b(\phi^n_{\eta_n(s)})|ds\leq C_{\a,L,H}\varepsilon. \nonumber
	\end{align}
	As a consequence, by Lemma \ref{Separation lemma} we get that
	\begin{align}\label{I_4}
		\limsup _{\varepsilon \rightarrow 0}\ \E(\exp(c\mathcal{J}^n_4)|\|B^H\|<\varepsilon)\leq1,
	\end{align}
	for every real number $c$.
	
	$\clubsuit \ \text{Term}$ \ $\mathcal{J}^n_1$
	
	Applying classical Taylor expansion to $b(\phi^{n}_{\eta_n(s)}+B_{\eta_n(s)}^H)$ at $\phi^n$ we have
	\begin{align}
		b(\phi^{n}_{\eta_n(s)}+B_{\eta_n(s)}^H)=b(\phi^n_s)+b'(\phi^n_{\eta_n(s)})B_{\eta_n(s)}^H+R_s^{n},\nonumber
	\end{align}
	where $R^{n}$ denotes the remainder term. If $\|B^H\|_{\b}\leq \varepsilon$, by Young's inequality we have that
	\begin{align}\label{the bound of remainder term}
		\|R^n\|_{\infty}\leq C_{L}\varepsilon^2.
	\end{align}
	We now rewrite the term $\mathcal{J}^n_1$.
	\begin{align}\label{combo of I_1}
		\mathcal{J}^n_1&=\int_{0}^{1}d_H^{-1}s^{-\alpha}\big(I^{\alpha}_{0+}u^{\alpha}b(\phi^{n}_{\eta_n(u)}+B_{\eta_n(u)}^H)\big)(s)dW_s\nonumber\\&=\int_{0}^{1}d_H^{-1}s^{-\alpha}\Big(I^{\alpha}_{0+}u^{\alpha}\big(b(\phi^n_s)+b'(\phi^n_{\eta_n(s)})B_{\eta_n(s)}^H+R_s^{n}\big)\Big)(s)dW_s\nonumber\\&:=\mathcal{J}^n_{11}+\mathcal{J}^n_{12}+\mathcal{J}^n_{13}.
	\end{align}
	Applying Theorem \ref{no random function} to $f=cd_H^{-1}s^{-\a}\big(I^{\a}_{0+}u^{\a}b(\phi^n_{\eta_n(u)})\big)(s)$ and Lemma \ref{Separation lemma}, we have
	\begin{align}\label{I_{11}}
		\limsup _{\varepsilon \rightarrow 0}\ \E(\exp(c\mathcal{J}^n_{11})|\|B^H\|_{\beta}<\varepsilon)\leq1,
	\end{align}
	for every real number $c$.
	
	To study the limit behavior of the conditional exponential moments of the term $\mathcal{J}^n_{12}$. We will express $\mathcal{J}^n_{12}$ as a double stochastic integral with respect to $W$ based on the integral representation of fractional Brownian motion $B^H$, so we have that
	\begin{align}
	d_H\mathcal{J}^n_{12}&=\int_{0}^{1}s^{-\alpha}\Big(I^{\alpha}_{0+}u^{\alpha}\big(b'(\phi^n_{\eta_n(u)})B^H_{\eta_n(u)}\big)\Big)(s)dW_s\nonumber\\&=\frac{1}{\G(\a)}\int_{0}^{1}s^{-\a}\int_{0}^{s}u^{\a}b'(\phi^n_{\eta_n(u)})B^H_{\eta_n(u)}(s-u)^{\a-1}dudW_s\nonumber\\&=\sum_{k=0}^{n-1}\frac{1}{\G(\a)}\int_{t_k}^{t_{k+1}}s^{-\a}\sum_{j=0}^{k-1}\int_{t_j}^{t_{j+1}}u^{\a}b'(\phi^n_{t_j})B^H_{t_j}(s-u)^{\a-1}dudW_s\nonumber\\&\quad+\sum_{k=0}^{n-1}\frac{1}{\G(\a)}\int_{t_k}^{t_{k+1}}s^{-\a}\int_{t_k}^{s}u^{\a}b'(\phi^n_{t_k})B^H_{t_k}(s-u)^{\a-1}dudW_s\nonumber\\&=\sum_{k=0}^{n-1}\frac{1}{\G(\a)}\int_{t_k}^{t_{k+1}}s^{-\a}\sum_{j=0}^{k-1}\int_{t_j}^{t_{j+1}}u^{\a}b'(\phi^n_{t_j})(s-u)^{\a-1}\int_{0}^{t_j}K^{H}(t_j,r)dW_rdudW_s\nonumber\\&\quad+\sum_{k=0}^{n-1}\frac{1}{\G(\a)}\int_{t_k}^{t_{k+1}}s^{-\a}\int_{t_k}^{s}u^{\a}b'(\phi^n_{t_k})(s-u)^{\a-1}\int_{0}^{t_k}K^{H}(t_k,r)dW_rdudW_s\nonumber\\&=\sum_{k=0}^{n-1}\frac{1}{\G(\a)}\int_{0}^{1}\int_{0}^{1}s^{-\a}\mathbb{I}_{s\in[t_k,t_{k+1})}\sum_{j=0}^{k-1}K^{H}(t_j,r)\int_{t_j}^{t_{j+1}}u^{\a}b'(\phi^n_{t_j})(s-u)^{\a-1}du\mathbb{I}_{r\leq t_j}dW_rdW_s\nonumber\\&\quad+\sum_{k=0}^{n-1}\frac{1}{\G(\a)}\int_{0}^{1}\int_{0}^{1}s^{-\a}\mathbb{I}_{s\in[t_k,t_{k+1})}K^{H}(t_k,r)\int_{t_k}^{s}u^{\a}b'(\phi^n_{t_k})(s-u)^{\a-1}du\mathbb{I}_{r\leq t_k}dW_rdW_s\nonumber\\&=\int_{0}^{1}\int_{0}^{1}f_{n}(s,r)dW_rdW_s=\int_{0}^{1}\int_{0}^{1}\tilde{f}_n(s,r)dW_rdW_s,\nonumber
	\end{align}
	where $\tilde{f}_n$ is the symmetrization of the function $f_n$ and 
	\begin{align}
	f_n(s,r)&=\sum_{k=0}^{n-1}\frac{1}{\Gamma(\a)}s^{-\a}\mathbb{I}_{s\in[t_k,t_{k+1})}\sum_{j=0}^{k-1}K^H(t_j,r)b'(\phi^n_{t_j})\int_{t_j}^{t_{j+1}}u^{\a}(s-u)^{\a-1}du\mathbb{I}_{r\leq t_j}\nonumber\\&\quad+\sum_{k=0}^{n-1}\frac{1}{\Gamma(\a)}s^{-\a}b'(\phi^n_{t_k})K^H(t_k,r)\mathbb{I}_{s\in[t_k,t_{k+1})}\int_{t_k}^{s}u^{\a}(s-u)^{\a-1}du\mathbb{I}_{r\leq t_k}\nonumber\\&:=\mathcal{M}^n_1(s,r)+\mathcal{M}^n_2(s,r).\nonumber
	\end{align}
	By Lemma \ref{singular reace class}, the operator $K(\tilde{f}^n)$ is nuclear. So the trace of this operator can be obtained as
	$$Tr\tilde{f}_n=\int_{0}^{1}\tilde{f}^n(s,s)ds=\frac{1}{2}\int_{0}^{1}f^n(s,s)ds.$$
	Note that the function $\tilde{f}_n$ is not continuous on the axes, but the result of \cite{Bal76} still holds in this case taking into account the particular form of the function $f_n$. We now compute the integral $\int_{0}^{1}f_n(s,s)ds$. 
	
	For the term $\mathcal{M}^n_1(s,r)$, we have the constraint $s\geq t_k, r\leq t_{j-1}$ for $j=0,\cdots,k-1$. If we set $r=s$, then $s\leq t_{j-1}\leq t_{k-1}<t_{k}\leq s$, which is impossible. Therefore,
	$$\mathcal{M}^n_1(s,s)=0.$$

	We now estimate the term $\mathcal M_2^n(s,u)$. Using the change of variables
	$v=\frac{u-t_k}{s-t_k}$, we obtain
	\begin{align}
		\int_{t_k}^{s}u^\alpha(s-u)^{\alpha-1}du
		&=
		(s-t_k)^\alpha
		\int_0^1
		\big(t_k+v(s-t_k)\big)^\alpha
		(1-v)^{\alpha-1}dv .
		\nonumber
	\end{align}
	Hence,
	\begin{align}
		\mathcal{M}^n_2(s,r)
		&=
		\sum_{k=0}^{n-1}
		s^{-\alpha}
		b'(\phi^n_{t_k})
		K^H(t_k,r)
		\mathbb{I}_{s\in[t_k,t_{k+1})}
		\mathbb{I}_{r\leq t_k}
		\nonumber\\
		&\quad\times
		(s-t_k)^\alpha
		\int_0^1
		\big(t_k+v(s-t_k)\big)^\alpha
		(1-v)^{\alpha-1}dv .
		\nonumber
	\end{align}
	Taking $r=s$, the indicator function gives $s\leq t_k$.	On the other hand, since $s\in[t_k,t_{k+1})$, we have $s\geq t_k$. Therefore, $s=r=t_k$.
	Consequently, $\mathcal{M}^n-2(s,s)$ can be nonzero only at the finite set of
	partition points
	\[
	\{t_k,\ k=0,\ldots,n-1\}.
	\]
	Since this set has Lebesgue measure zero, we obtain
	\[
	\int_0^1 \mathcal{M}^n_2(s,s)\,ds=0 .
	\]
	To summarized what we have proved, Lemma \ref{Separation lemma} gives us
	\begin{align}\label{I_{13}}
		\limsup _{\varepsilon \rightarrow 0}\ \E(\exp(\mathcal{J}^n_{12})|\|B^H\|<\varepsilon)=1.
	\end{align}
	Finally, it only remains to study the limit behavior of the term $\mathcal{J}^n_{13}$. For any $c\in \mathbb{R}$ and $\delta>0$ we can write
	\begin{align}
	\E\left(\!\exp(c\mathcal{J}^n_{13})|\|B^H\|\!\leq \varepsilon\!\right)
	\!\leq e^{\delta}
	+\int_{\delta}^{\infty}\!\!e^{\xi}
	\mathbb{P}\left(\!|c\mathcal{J}^n_{13}|>\xi|\|B^H\|_{\b}\!\leq \varepsilon\!\right)d \xi 
	+e^{\delta}\mathbb{P}\left(\!|c\mathcal{J}^n_{13}|>\delta |\|B^H\|_{\b}\!\leq \varepsilon\!\right).\nonumber
\end{align}
	Define the martingale $M^n_t=c\int_{0}^{t}d_H^{-1}s^{-\a}\Big(I^{\alpha}_{0+}u^{\alpha}R_u^{n}\Big)(s)dW_s$ whose quadratic variations can be estimated by \eqref{the bound of remainder term} as follows
	\begin{align}
	\langle M^n \rangle_t=c^2\int_{0}^{t}\Big(d_H^{-1}s^{-\alpha}\Big(I^{\alpha}_{0+}u^{\alpha}R_u^{n}\Big)(s)\Big)^2ds
	\leq \frac{c^2 C_{L,H}^2\b(\a,\a+1)^2}{(1+2\a)(\b(\a))^2}\varepsilon^4=C_{\a,L,H}\varepsilon^4.
\end{align}
	Applying the exponential inequality for martingales, we have
	\begin{align}\label{exponential inequality for martingales1}
		\mathbb{P}\Big(\Big|c\int_{0}^{t}d_H^{-1}s^{-\a}\Big(I^{\alpha}_{0+}u^{\alpha}R_u^{n}\Big)(s)dW_s\Big|>\xi, \|B^H\|_{\b}\leq \varepsilon\Big)\leq \exp\Big(-\frac{\xi^2}{2C_{\a,L,H}\varepsilon^{4}}\Big),
	\end{align}
	for every real number $c$.\\
	Combining Lemma \ref{small ball of holder} and inequality \eqref{exponential inequality for martingales1}, we see that
	\begin{align}
			\mathbb{P}\Big(\Big|c\int_{0}^{1}d_H^{-1}s^{-\a}\Big(I^{\alpha}_{0+}u^{\alpha}R_u^{n}\Big)(s)dW_s\Big|>\xi\Big|\|B^H\|_{\b}\leq \varepsilon\Big)
			\leq \exp\left(-\frac{\xi^2}{2C_{\a,L,H}\varepsilon^{4}}\right)\exp\big(	C_H \varepsilon^{-\frac{1}{H-\b}}\big).\nonumber
		\end{align}
	Using the latter estimate we have for every $\delta>0$ and every $0<\varepsilon<1$
	\begin{align}
	\E&\Big(\exp(c\mathcal{J}^n_{13})|\|B^H\|_{\beta}\leq \varepsilon\Big)\nonumber\\
	&\leq e^{\delta}+\int_{\delta}^{\infty}\exp\Big\{\xi-\frac{\xi^2}{2C_{\a,L,H}\varepsilon^{4}}+C_H \varepsilon^{-\frac{1}{H-\b}}\Big\}d \xi 
	+\exp\Big\{\delta-\frac{\delta^2}{2C_{\a,L,H}\varepsilon^{4}}+C_H \varepsilon^{-\frac{1}{H-\b}}\Big\}.\qquad\quad\;\nonumber
\end{align}
	Letting $\varepsilon$ and then $\delta$ tend to zero, we obtain
	\begin{align}\label{I_{14}}
		\limsup _{\varepsilon \rightarrow 0}\ \E(\exp(c\mathcal{J}^n_{13})|\|B^H\|_{\b}<\varepsilon)\leq1,
	\end{align}
	for every real number $c$. \\
	Finally, we can summarize what we have derived, \eqref{small all},  \eqref{I_2}, \eqref{I_3}, \eqref{I_4}, \eqref{combo of I_1}, \eqref{I_{11}}, \eqref{I_{13}} and \eqref{I_{14}} give us
	$$
	\lim _{\varepsilon \rightarrow 0} \frac{\mathbb{P}(\|X^n-\phi^{n}\|_{\b}<\varepsilon)}{\mathbb{P}(\|B^H\|_{\b}<\varepsilon)}=\exp \left(-\frac{1}{2}\int_{0}^{1}\big|\dot{\phi}^{n}_s-d_H^{-1}s^{-\a}\big(I^{\a}_{0+}u^{\a}b(\phi^n_{\eta_n(u)})\big)(s)\big|^2ds\right).
	$$
	The proof of Theorem \ref{th:singular case} is complete.
\end{proof}
\begin{thm}
	\label{th:regular case}
	Let $X^n$ be the solution of \eqref{Euler-main equation} with Hurst index $\frac{1}{2}<H<1$. Let $\{\phi^n\}_{n\in\mathbb{N}}$ be a family of numerical reference paths such that $\phi^n-x\in\mathcal{H}^2$ and assume $b\in C^2_b(\mathbb{R})$. Then the Onsager-Machlup action functional of $X^n$ for the norms $\|\cdot\|_{\b}$  with $H-\frac{1}{2}<\beta<H-\frac{1}{4}$ exists and is given by
	$$L^n(\phi^n,\dot{\phi}^n)=-\frac{1}{2}\int_{0}^{1}\big|\dot{\phi}^{n}_s-d_H^{-1}s^{\a}\big(D^{\a}_{0+}u^{-\a}b(\phi^n_{\eta_n(u)})\big)(s)\big|^2ds,$$
	where $\a=H-\frac{1}{2}$ and numerical reference path $\phi^n,n\in\mathbb{N}$ is the function such that $K^H\dot{\phi}^n=\phi^n-x$.
\end{thm}
\begin{proof}
	Recall operator $(K^H)^{-1}$ is defined by
	$$\Big(\big(K^H\big)^{-1}h\Big)(s)=d_H^{-1}s^{\alpha}\big(D^{\alpha}_{0+}u^{-\alpha}h'\big)(s),$$
	where $\a=H-\frac{1}{2}$ when $H>\frac{1}{2}.$ So we can rewrite small probability \eqref{all} as
	\begin{align}\label{small all 2}
		&\mathbb{P}(\|X^n-\phi^{n}\|_{\beta}\leq\varepsilon)\nonumber\\
		&=\E\left(\exp(\mathcal{I}^n_1+\mathcal{I}^n_2+\mathcal{I}^n_3+\mathcal{I}^n_4)\mathbb{I}_{\|B^H\|\leq\varepsilon}\right)
		\times\exp\left(-\frac{1}{2}\int_{0}^{1}\big|\dot{\phi}^{n}_s-d_H^{-1}s^{\a}\big(D^{\a}_{0+}u^{-\a}b(\phi^n_{\eta_n(u)})\big)(s)\big|^2ds\right),
	\end{align}
	where
	\begin{align}
		\mathcal{I}^n_1&=\int_{0}^{1}d_H^{-1}s^{\alpha}\big(D^{\alpha}_{0+}u^{-\alpha}b(\phi^{n}_{\eta_n(u)}+B_{\eta_n(u)}^H)\big)(s)dW_s,\nonumber\\
		\mathcal{I}_2&=\int_{0}^{1}-\dot{\phi}^{n}_sdW_s,\nonumber\\
		\mathcal{I}^n_3&=\int_{0}^{1}d_H^{-1}\dot{\phi}^{n}_s\cdot\Big(s^{\alpha}\big(D^{\alpha}_{0+}u^{-\alpha}\big(b(\phi^{n}_{\eta_n(u)}+B_{\eta_n(u)}^H)-b(\phi^n_{\eta_n(u)})\big)\big)(s)\Big)ds,\nonumber\\\mathcal{I}^n_4&=\frac{1}{2}\int_{0}^{1}\Bigg(\Big(s^{\alpha}\big(D^{\alpha}_{0+}u^{\alpha}b(\phi^n_{\eta_n(u)})\big)(s)\Big)^2-\Big(s^{\a}\big(D^{\alpha}_{0+}u^{\alpha}b(\phi^{n}_{\eta_n(u)}+B_{\eta_n(u)}^H)\big)(s)\Big)^2\Bigg)ds.\nonumber
	\end{align}
	Then, we could deal with each term independently by applying Lemma \ref{Separation lemma}.
	
	$\clubsuit \ \text{Term}$ \ $\mathcal{I}^n_2$
	
	Applying Theorem \ref{no random function} to $f=-c\dot{\phi}^{n}_s$ and Lemma \ref{norm}, we have
	\begin{align}\label{J_2}
		\limsup _{\varepsilon \rightarrow 0}\ \E(\exp(c\mathcal{I}^n_2)|\|B^H\|_{\beta}<\varepsilon)\leq1,
	\end{align}
	for every real number $c$.
	
	$\clubsuit \ \text{Term}$ \ $\mathcal{I}^n_3$
	
	Using the Weyl representation \eqref{weyl representation} for the fractional derivative, we have that
	\begin{align}
		\Big|&s^{\alpha}\big(D^{\alpha}_{0+}u^{-\alpha}\big(b(\phi^{n}_{\eta_n(u)}+B_{\eta_n(u)}^H)-b(\phi^n_{\eta_n(u)})\big)\big)(s)\Big|\nonumber\\&=\frac{1}{\G(1-\a)}\Bigg|\frac{b(\phi^{n}_{\eta_n(s)}+B_{\eta_n(s)}^H)-b(\phi^n_{\eta_n(s)})\big)}{s^{\a}}\nonumber\\&\quad+\a s^{\a}\int_{0}^{s}\frac{s^{-\a}\big(b(\phi^{n}_{\eta_n(s)}+B_{\eta_n(s)}^H)-b(\phi^n_{\eta_n(s)})-r^{-\a}\big(b(\phi^{n}_{\eta_n(r)}+B_{\eta_n(r)}^H)-b(\phi^n_{\eta_n(r)})\big)}{(s-r)^{\a+1}}dr\Bigg|\nonumber\\&\leq \mathcal{I}^n_{31}+\mathcal{I}^n_{32}+\mathcal{I}^n_{33},\nonumber
	\end{align}
	where
	\begin{align}
		\mathcal{I}^n_{31}&=\frac{1}{\G(1-\a)}\Big|\frac{b(\phi^{n}_{\eta_n(s)}+B_{\eta_n(s)}^H)-b(\phi^n_{\eta_n(s)})\big)}{s^{\a}}\Big|,\nonumber\\\mathcal{I}^n_{32}&=\frac{\a}{\Gamma(1-\a)}\Bigg|\int_{0}^{s}\frac{\big(b(\phi^{n}_{\eta_n(s)}+B_{\eta_n(s)}^H)-b(\phi^n_{\eta_n(s)})-\big(b(\phi^{n}_{\eta_n(r)}+B_{\eta_n(r)}^H)-b(\phi^n_{\eta_n(r)})\big)}{(s-r)^{\a+1}}dr\Bigg|,\nonumber\\\mathcal{I}^n_{33}&=\frac{\a s^{\a}}{\G(1-\a)}\Bigg|\int_{0}^{s}\frac{s^{-\a}-r^{-\a}}{(s-r)^{\a+1}}\big(b(\phi^{n}_{\eta_n(r)}+B_{\eta_n(r)}^H)-b(\phi^n_{\eta_n(r)})\big)dr\Bigg|.\nonumber
	\end{align}
	So under the condition $\|B^H\|_{\beta}<\varepsilon$, note that $\eta_n(s)\leq s, s\in [0,1]$ and using the fact that $b$ is Lipschitz continuous and bounded with constant $L$, we have that ($\b>\a$)
	\begin{align}\label{J_{31}}
		\mathcal{I}^n_{31}\leq \frac{L}{\Gamma(1-\alpha)}\frac{\eta_n(s)^{\b}}{s^{\a}}\|B\|_{\b}\leq C_{\a,\b,L}\varepsilon.\qquad\qquad
	\end{align}
	On the other hand, from ($v=\frac{r}{s}$)
	\begin{align}\label{integral transform}
		\int_{0}^{s}\frac{r^{-\a}-s^{-\a}}{(s-r)^{\a+1}}dr=s^{-2\a}\int_{0}^{1}\frac{v^{-\a}-1}{(1-v)^{\a+1}}dv:=C_{\a}s^{-2\a},
	\end{align}
	where
	$C_{\a}$ is a constant depending on $\a$. So we easily obtain ($\eta_n(r)\leq r\leq s$)
	\begin{align}\label{J_{33}}
		\mathcal{I}^n_{33}&\leq \frac{\a s^{\a}L}{\G(1-\a)}\int_{0}^{s}\frac{r^{-\a}-s^{-\a}}{(s-r)^{\a+1}}\big|B_{\eta_n(r)}^H\big|dr\nonumber\\&\leq \frac{\a s^{\a}L}{\G(1-\a)}\int_{0}^{s}\frac{r^{-\a}-s^{-\a}}{(s-r)^{\a+1}}|\eta_n(r)|^{\b}dr\|B^H\|_{\b}\nonumber\\&\leq \frac{C_{\a}\a s^{\b-\a}L}{\G(1-\a)}\|B^H\|_{\b}\leq C_{\a,\b,L}\varepsilon.
	\end{align}
	It remains to study the limit behavior of the term $\mathcal{I}^n_{32}$. To begin with, we give an integral equality.
	\begin{align}
		&\big(b(\phi^{n}_{\eta_n(s)}+B_{\eta_n(s)}^H)-b(\phi^n_{\eta_n(s)})-\big(b(\phi^{n}_{\eta_n(r)}+B_{\eta_n(r)}^H)-b(\phi^n_{\eta_n(r)})\big)\nonumber\\
		&=\int_{0}^{1}b'(\lambda B_{\eta_n(s)}^H+\phi^n_{\eta_n(s)})d\lambda  B^H_{\eta_n(s)}-\int_{0}^{1}b'(\lambda B_{\eta_n(r)}^H+\phi^n_{\eta_n(r)})d\lambda  B^H_{\eta_n(r)}\qquad\nonumber\\
		&=\int_{0}^{1}\Big(b'(\lambda B_{\eta_n(s)}^H+\phi^n_{\eta_n(s)})-b'(\lambda B_{\eta_n(r)}^H+\phi^n_{\eta_n(r)})\Big)d\lambda \cdot B^H_{\eta_n(s)}\nonumber\\
		&\quad+\int_{0}^{1}b'(\lambda B_{\eta_n(r)}^H+\phi^n_{\eta_n(r)})d\lambda\Big(B^H_{\eta_n(s)}-B^H_{\eta_n(r)}\Big).\nonumber
	\end{align}
	Then we shall further use that $b'\in C^2_b(\mathbb{R}$ are Lipschitz with constant $L'$, and $\phi^n$ is $H$-H$\mathrm{\ddot{o}}$lder continuous yields
	\begin{align}
		&\big(b(\phi^{n}_{\eta_n(s)}+B_{\eta_n(s)}^H)-b(\phi^n_{\eta_n(s)})-\big(b(\phi^{n}_{\eta_n(r)}+B_{\eta_n(r)}^H)-b(\phi^n_{\eta_n(r)})\big)\nonumber\\&\leq \frac{L'}{2}\big|B_{\eta_n(s)}^H-B_{\eta_n(r)}^H\big|\big|B_{\eta_n(s)}^H\big|+L'|\phi^n_{\eta_n(s)}-\phi^n_{\eta_n(r)}|\big|B_{\eta_n(s)}^H\big|\nonumber\\&\quad+\|b'\|_{\infty}\big|B_{\eta_n(s)}^H-B_{\eta_n(r)}^H\big|\nonumber\\&\leq \frac{L'}{2}\|B^H\|_{\b}^2|\eta_n(s)-\eta_n(r)|^{\b}|\eta_n(s)|^{\b}+L'\|\phi^n\|_{H}|\eta_n(s)-\eta_n(r)|^{H}|\eta_n(s)|^{\b}\nonumber\\&\quad+\|b'\|_{\infty}\|B^H\|_{\b}|\eta_n(s)-\eta_n(r)|^{\b}\nonumber.
	\end{align}
	Therefore we have 
	\begin{align}\label{J_{32}}
		\mathcal{I}^n_{32}&=\frac{\a}{\Gamma(1-\a)}\Bigg|\int_{0}^{s}\frac{\big(b(\phi^{n}_{\eta_n(s)}+B_{\eta_n(s)}^H)-b(\phi^n_{\eta_n(s)})-\big(b(\phi^{n}_{\eta_n(r)}+B_{\eta_n(r)}^H)-b(\phi^n_{\eta_n(r)})\big)}{(s-r)^{\a+1}}dr\Bigg|\nonumber\\&\leq C_{\a,L'}\Big[\int_0^s\frac{|\eta_n(s)-\eta_n(r)|^{\b}|\eta_n(s)|^{\b}+|\eta_n(s)-\eta_n(r)|^{H}|\eta_n(s)|^{\b}+|\eta_n(s)-\eta_n(r)|^{\b}}{(s-r)^{\a+1}}dr\Big]\|B^H\|_{\b}\nonumber\\&\leq C_{\a,\b,L'}\int_0^s \frac{|\eta_n(s)-\eta_n(r)|^{H}}{(s-r)^{\alpha+1}}dr\varepsilon\nonumber\\&\leq C_{\a,\b,H,n,L'}\varepsilon,
	\end{align}
	where the last inequality above utilizes the following inequality. If $s,r\in[t_{k-1},t_{k})$ for some $k\in\{1,\cdots,n-1\}$, $\eta_n(s)=\eta_n(r)=t_{k-1}=0$, so we only need to consider the case $r\in[t_{l-1},t_{l})$ for some $l\in\{1,\cdots,n-1\}$ and $l\neq k$, we have that
	\begin{align}
		&\int_0^s \frac{|\eta_n(s)-\eta_n(r)|^{H}}{(s-r)^{H+1/2}}d r\nonumber\\
		&=\int_0^s \frac{|t_{k-1}-t_{l-1}|^{H}}{(s-r)^{H+1/2}}d r\nonumber\\
		&\leq \int_0^s \Big|\frac{t_{k-1}-t_{l-1}}{s-t_{l-1}}\Big|^{H}(s-r)^{-1/2}d r\nonumber\\
		&\leq C\int_0^s(s-r)^{-1/2}d r<\infty.\nonumber
	\end{align}
	Hence, it follows from estimates \eqref{J_{31}}, \eqref{J_{33}} and \eqref{J_{32}} that
	\begin{align}\label{the following need}
		\Big|s^{\alpha}\big(D^{\alpha}_{0+}u^{-\alpha}\big(b(\phi^{n}_{\eta_n(u)}+B_{\eta_n(u)}^H)-b(\phi_{\eta_n(u)})\big)\big)(s)\Big|\leq C_{\a,\b,H,n,L'}\varepsilon.
	\end{align}
	Therefore,
	\begin{align}\label{J_3}
		\limsup _{\varepsilon \rightarrow 0}\ \E(\exp(c\mathcal{I}^n_3)|\|B^H\|_{\beta}<\varepsilon)\leq1,
	\end{align}
	for every real number $c$.
	
	$\clubsuit \ \text{Term}$ \ $\mathcal{I}^n_4$\\
	By inequality $|a^2-b^2|\leq (a-b)^2+2|(a-b)b|$, we have
	\begin{align}
		|\mathcal{I}^n_4|&=\frac{1}{2}\int_{0}^{1}\Bigg|\Big(d_H^{-1}s^{\alpha}\big(D^{\alpha}_{0+}u^{\alpha}b(\phi_{\eta_n(u)})\big)(s)\Big)^2-\Big(d_H^{-1}s^{\a}\big(D^{-\alpha}_{0+}u^{\alpha}b(\phi^{n}_{\eta_n(u)}+B_{\eta_n(u)}^H)\big)(s)\Big)^2\Bigg|ds\nonumber\\&\leq \frac{1}{2}\int_{0}^{1}\Big(d_H^{-1}s^{\a}\big(D^{\alpha}_{0+}u^{\alpha}(b(\phi^{n}_{\eta_n(u)}+B_{\eta_n(u)}^H)-b(\phi^n_{\eta_n(u)}))\big)(s)\Big)^2ds\nonumber\\&\quad+\int_{0}^{1}\Big|d_H^{-2}\Big(s^{\a}\big(D^{\alpha}_{0+}u^{\alpha}(b(\phi^{n}_{\eta_n(u)}+B_{\eta_n(u)}^H)-b(\phi^n_{\eta_n(u)}))\big)(s)\Big)s^{\a}D^{\a}_{0+}s^{-\a}b(\phi^n_{\eta_n(u)})(s)\Big|ds.\nonumber
	\end{align}
	Using \eqref{the following need} it is easy to see that
	$$|\mathcal{I}^n_4|\leq C_{\a,\b,H,L'}\varepsilon.$$
	As a consequence,
	\begin{align}\label{J_4}
		\limsup _{\varepsilon \rightarrow 0}\ \E(\exp(c\mathcal{I}^n_4)|\|B^H\|_{\beta}<\varepsilon)\leq1,
	\end{align}
	for every real number $c$.
	
	$\clubsuit \ \text{Term}$ \ $\mathcal{I}^n_1$\\
	Applying classical Taylor expansion to $b(\phi^{n}_{\eta_n(s)}+B_{\eta_n(s)}^H)$ at $\phi^n$ we have
	\begin{align}
		b(\phi^{n}_{\eta_n(s)}+B_{\eta_n(s)}^H)=b(\phi^n_{\eta_n(s)})+b'(\phi^n_{\eta_n(s)})B_{\eta_n(s)}^H+R_s^{n},\nonumber
	\end{align}
	where $R^{n}$ denotes the remainder term. If $\|B^H\|_{\b}\leq \varepsilon$, by Young's inequality we have that
	\begin{align}\label{estimate of Remainder term}
		\|R^{n}\|_{\infty}\leq C_{\b}\varepsilon^2.
	\end{align}
	Hence, we can rewrite the term $\mathcal{I}^n_1$.
	\begin{align}\label{the part of J_1}
		\mathcal{I}^n_1&=\int_{0}^{1}d_H^{-1}s^{\alpha}\big(D^{\alpha}_{0+}u^{-\alpha}b(\phi^{n}_{\eta_n(u)}+B_{\eta_n(u)}^H)\big)(s)dW_s\nonumber\\&=\int_{0}^{1}d_H^{-1}s^{\alpha}\Big(D^{\alpha}_{0+}u^{-\alpha}\big(b(\phi^n_{\eta_n(u)})+b'(\phi^n_{\eta_n(u)})B_{\eta_n(u)}^H+R_u^{n}\big)\Big)(s)dW_s\nonumber\\&:=\mathcal{I}^n_{11}+\mathcal{I}^n_{12}+\mathcal{I}^n_{13}.
	\end{align}
	Applying Theorem \ref{no random function} to $f=cd_H^{-1}s^{\a}\big(D^{\a}_{0+}u^{-\a}b(\phi^n_{\eta_n(u)})\big)(s)$ and Lemma \ref{norm}, we have
	\begin{align}\label{J_{11}}
		\limsup _{\varepsilon \rightarrow 0}\ \E(\exp(c\mathcal{I}^n_{11})|\|B^H\|_{\beta}<\varepsilon)\leq1,
	\end{align}
	for every real number $c$.\\
	To study the limit behavior of the conditional exponential moments of the term $\mathcal{I}^n_{12}$. We will express $\mathcal{I}^n_{12}$ as a double stochastic integral with respect to $W$ based on the Weyl representation of the fractional derivative and the integral representation of fractional Brownian motion $B^H$, so we have
	\begin{align}
		d_H\mathcal{I}^n_{12}&=\int_{0}^{1}s^{\alpha}\Big(D^{\alpha}_{0+}u^{-\alpha}\big(b'(\phi^n_{\eta_n(u)})B^H_{\eta_n(u)}\big)\Big)(s)dW_s\nonumber\\&=\frac{1}{\G(1-\a)}\!\int_{0}^{1}\!\Big(s^{-\a}b'(\phi^n_{\eta_n(s)})B_{\eta_n(s)}^H\!+\!\a s^{\a}\!\int_{0}^{s}\!\frac{s^{-\a}b'(\phi^n_{\eta_n(s)})B_{\eta_n(s)}^H-r^{-\a}b'(\phi^n_{\eta_n(r)})B_{\eta_n(r)}^H}{(s-r)^{\a+1}}dr\!\Big)dW_s\nonumber\\&=\sum_{k=0}^{n-1}\frac{1}{\G(1-\a)}\int_{t_k}^{t_{k+1}}\Big(s^{-\a}b'(\phi^n_{t_k})B_{t_k}^H+\a s^{\a}\sum_{j=0}^{k-1}\int_{t_j}^{t_{j+1}}\frac{s^{-\a}b'(\phi^n_{t_k})B_{t_k}^H-r^{-\a}b'(\phi^n_{t_j})B_{t_j}^H}{(s-r)^{\a+1}}dr\nonumber\\&\quad+\a s^{\a}\int_{t_k}^{s}\frac{s^{-\a}b'(\phi^n_{t_k})B_{t_k}^H-r^{-\a}b'(\phi^n_{t_k})B_{t_k}^H}{(s-r)^{\a+1}}dr\Big)dW_s\nonumber\\&=\sum_{k=0}^{n-1}\frac{1}{\G(1-\a)}\int_{t_k}^{t_{k+1}}\Big(s^{-\a}b'(\phi^n_{t_k})\int_0^{t_k}K^H(t_k,u)dW_u+\a s^{\a}\nonumber\\&\quad\cdot\sum_{j=0}^{k-1}\int_{t_j}^{t_{j+1}}\frac{s^{-\a}b'(\phi^n_{t_k})\int_0^{t_k}K^H(t_k,u)dW_u-r^{-\a}b'(\phi^n_{t_j})\int_0^{t_j}K^H(t_j,u)dW_u}{(s-r)^{\a+1}}dr\nonumber\\&\quad+\a s^{\a}b'(\phi^n_{t_k})\int_0^{t_k}K^H(t_k,u)dW_u\int_{t_k}^{s}\frac{s^{-\a}-r^{-\a}}{(s-r)^{\a+1}}dr\Big)dW_s\nonumber\\&=\int_{0}^{1}\int_{0}^{1}\sum_{k=0}^{n-1}\frac{1}{\G(1-\a)}\mathbb{I}_{s\in[t_k,t_{k+1})}\Big(s^{-\a}b'(\phi^n_{t_k})K^H(t_k,u)\mathbb{I}_{u\leq t_k}+\a s^{\a}\nonumber\\&\quad\cdot\sum_{j=0}^{k-1}\int_{t_j}^{t_{j+1}}\frac{s^{-\a}b'(\phi^n_{t_k})K^H(t_k,u)\mathbb{I}_{u\leq t_k}-r^{-\a}b'(\phi^n_{t_j})K^H(t_j,u)\mathbb{I}_{u\leq t_j}}{(s-r)^{\a+1}}dr\nonumber\\&\quad+\a s^{\a}b'(\phi^n_{t_k})K^H(t_k,u)\mathbb{I}_{u\leq t_k}\int_{t_k}^{s}\frac{s^{-\a}-r^{-\a}}{(s-r)^{\a+1}}dr\Big)dW_udW_s\nonumber\\&=\int_{0}^{1}\int_{0}^{1}f_{n}(s,u)dW_udW_s=\int_{0}^{1}\int_{0}^{1}\tilde{f}_n(s,u)dW_udW_s,\nonumber
	\end{align}
	where $\tilde{f}$ is the symmetrization of the function ($\tilde f(s,u)=\tilde f(u,s)$)
	\begin{align}
		f_n(s,u)&=\sum_{k=0}^{n-1}\frac{1}{\G(1-\a)}\mathbb{I}_{s\in[t_k,t_{k+1})}\Bigg[s^{-\a}b'(\phi^n_{t_k})K^H(t_k,u)\mathbb{I}_{u\leq t_k}+\a s^{\a}\nonumber\\
		&\quad\cdot\sum_{j=0}^{k-1}\int_{t_j}^{t_{j+1}}\frac{s^{-\a}b'(\phi^n_{t_k})K^H(t_k,u)\mathbb{I}_{u\leq t_k}-r^{-\a}b'(\phi^n_{t_j})K^H(t_j,u)\mathbb{I}_{u\leq t_j}}{(s-r)^{\a+1}}dr\qquad\qquad\quad\qquad\nonumber\\
		&\quad+\a s^{\a}b'(\phi^n_{t_k})K^H(t_k,u)\mathbb{I}_{u\leq t_k}\int_{t_k}^{s}\frac{s^{-\a}-r^{-\a}}{(s-r)^{\a+1}}dr\Bigg].\nonumber
	\end{align}
	By Lemma \ref{regular trace class}, the operator $K(\tilde{f}_n)$ is nuclear. So the trace of this operator can be obtained as
	$$Tr\tilde{f}_n=\int_{0}^{1}\tilde{f}_n(s,s)ds=\frac{1}{2}\int_{0}^{1}f_n(s,s)ds.$$
	In order to compute the integral $\int_{0}^{1}f_n(s,s)ds.$ Let us rewrite $f_n(s,u)$ as
	\begin{align}
		f_n(s,u)&=\sum_{k=0}^{n-1}\frac{1}{\G(1-\a)}\mathbb{I}_{s\in[t_k,t_{k+1})}\Bigg[s^{-\a}b'(\phi^n_{t_k})K^H(t_k,u)\mathbb{I}_{u\leq t_k}+\a s^{\a}\nonumber\\&\quad\cdot\sum_{j=0}^{k-1}\int_{t_j}^{t_{j+1}}\frac{s^{-\a}b'(\phi^n_{t_k})K^H(t_k,u)\mathbb{I}_{u\leq t_k}-r^{-\a}b'(\phi^n_{t_j})K^H(t_j,u)\mathbb{I}_{u\leq t_j}}{(s-r)^{\a+1}}dr\nonumber\\&\quad+\a s^{\a}b'(\phi^n_{t_k})K^H(t_k,u)\mathbb{I}_{u\leq t_k}\int_{t_k}^{s}\frac{s^{-\a}-r^{-\a}}{(s-r)^{\a+1}}dr\Bigg]\nonumber\\&=\sum_{k=0}^{n-1}\frac{1}{\G(1-\a)}\mathbb{I}_{s\in[t_k,t_{k+1})}s^{-\a}b'(\phi^n_{t_k})K^H(t_k,u)\mathbb{I}_{u\leq t_k}\nonumber\\&\quad+\sum_{k=0}^{n-1}\frac{1}{\G(1-\a)}\mathbb{I}_{s\in[t_k,t_{k+1})}\a s^{\a}\sum_{j=0}^{k-1}\int_{t_j}^{t_{j+1}}\frac{s^{-\a}b'(\phi^n_{t_k})K^H(t_k,u)\mathbb{I}_{u\leq t_k}-r^{-\a}b'(\phi^n_{t_j})K^H(t_j,u)\mathbb{I}_{u\leq t_j}}{(s-r)^{\a+1}}dr\nonumber\\&\quad+\sum_{k=0}^{n-1}\frac{1}{\G(1-\a)}\mathbb{I}_{s\in[t_k,t_{k+1})}\a s^{\a}b'(\phi^n_{t_k})K^H(t_k,u)\mathbb{I}_{u\leq t_k}\int_{t_k}^{s}\frac{s^{-\a}-r^{-\a}}{(s-r)^{\a+1}}dr\nonumber\\&:=\mathcal{P}^n_1(s,u)+\mathcal{P}^n_2(s,u)+\mathcal{P}^n_3(s,u).\nonumber
	\end{align}
	For $s\in[t_k,t_{k+1})$, in the expression of $\mathcal{P}^n_1(s,u)$, we have the constraint $u\leq t_k$. If we set $u=s$, then $u\leq t_k$ implies $s\leq t_k$. On the other hand, since $s\in[t_k,t_{k+1})$, we also have $s\geq t_k$. Hence $s=t_k$, and thus $u=s=t_k$. Therefore,
	 $$\mathcal{P}^n_1(s,s)=\sum_{k=0}^{n-1}\frac{1}{\G(1-\a)}s^{-\a}b'(\phi^n_{t_k})K^H(t_k,t_k).$$
	Since $K^H(t,t)=0$ for the Volterra regular kernel of fractional Brownian motion, we conclude $\mathcal{P}^n_1(s,s)=0$.
	
	We now estimate the term $\mathcal P_2^n(s,u)$. Define
	$$
	A_k(u):=
	b'(\phi^n_{t_k})
	K^H(t_k,u)
	\mathbb I_{\{u\le t_k\}}.
	$$
	Then we can rewrite $\mathcal P_2^n(s,u)$ as
	\begin{align}\label{dep of Pc2}
		\mathcal{P}^n_2(s,u)&=\sum_{k=0}^{n-1}\frac{1}{\G(1-\a)}\mathbb{I}_{s\in[t_k,t_{k+1})}\ s^{\a}\sum_{j=0}^{k-1}\int_{t_j}^{t_{j+1}}\frac{s^{-\a}A_k(u)-r^{-\a}A_j(u)}{(s-r)^{\a+1}}dr\qquad\qquad\qquad\quad\nonumber\\
		&=\sum_{k=0}^{n-1}\frac{1}{\G(1-\a)}\mathbb{I}_{s\in[t_k,t_{k+1})}\ s^{\a}\sum_{j=0}^{k-1}\int_{t_j}^{t_{j+1}}\frac{r^{-\a}(A_k(u)-A_j(u))}{(s-r)^{\a+1}}dr\nonumber\\&\quad+\sum_{k=0}^{n-1}\frac{1}{\G(1-\a)}\mathbb{I}_{s\in[t_k,t_{k+1})}\ s^{\a}\sum_{j=0}^{k-1}\int_{t_j}^{t_{j+1}}\frac{(s^{-\a}-r^{-\a})A_k(u)}{(s-r)^{\a+1}}dr\nonumber\\&:=\mathcal{P}^n_{21}(s,u)+\mathcal{P}^n_{22}(s,u).
	\end{align}
	We first estimate the term $\mathcal P_{21}^n$. Note that
	\begin{align}
		|\mathcal P_{21}^n(s,u)|
		&\le
		C_{b'}
		\sum_{k=0}^{n-1}
		\mathbb  I_{[t_k,t_{k+1})}(s)
		\sum_{j=0}^{k-1}\Big|K^H(t_k,u)\mathbb{I}_{u\leq t_k}-K^H(t_j,u)\mathbb{I}_{u\leq t_j}\Big|
		\int_{t_j}^{t_{j+1}}
		\frac{r^{-\a}}{(s-r)^{\alpha+1}}dr\nonumber\\&\le C_{\a,b'}
		\sum_{k=0}^{n-1}
		\mathbb  I_{[t_k,t_{k+1})}(s)
		\sum_{j=0}^{k-1}\Big|K^H(t_k,u)\mathbb{I}_{u\leq t_k}-K^H(t_j,u)\mathbb{I}_{u\leq t_j}\Big|.\nonumber
	\end{align}
Since $s\in[t_k,t_{k+1})$, we have $t_j<t_k\leq s, j=0,\cdots,k-1$. 
Hence if we set $s=u$, $\mathbb{I}_{s\leq t_j}=0$. Morevoer, under the constraints $s\in [t_k,t_{k+1})$ and $s\leq t_k$, we have $s=u=t_k$, then $K^H(t_k,t_k)=0$. Consequently,
$$\mathcal{P}^n_{21}(s,s)=0.$$	
For the term $\mathcal{P}^n_{22}$. Recall
\begin{align}\label{small ineq}
	|s^{-\alpha}-r^{-\alpha}|
	\le
	C_\alpha (s-r)r^{-\alpha-1},
\end{align}
	we obtain
	\begin{align}
		|\mathcal P_{22}^n(s,u)|
		&\le
		C_{\alpha,b'}
		\sum_{k=0}^{n-1}
		\mathbb I_{[t_k,t_{k+1})}(s)
		\sum_{j=0}^{k-1}
		\int_{t_j}^{t_{j+1}}
		(s-r)^{-\alpha}
		r^{-\alpha-1}
		|K^H(t_k,u)|
		\mathbb I_{\{u\le t_k\}}
		\,dr.
		\nonumber
	\end{align}
	Since
	\begin{align}
		\sum_{j=0}^{k-1}
	\int_{t_j}^{t_{j+1}}
	(s-r)^{-\alpha}
	r^{-\alpha-1}dr
	&\leq \sum_{j=0}^{k-1}u^{-\alpha-1}\int_{t_j}^{t_{j+1}}
	(s-r)^{-\alpha}dr\qquad\qquad\qquad\qquad\nonumber\\
	&\leq \sum_{j=0}^{k-1}u^{-\alpha-1}\Big((s-t_j)^{1-\a}-(s-t_{j+1})^{1-\a}\Big)\nonumber\\&\leq C_k u^{-\alpha-1}
s^{1-\a},\nonumber
	\end{align}
	it follows that
	\[
	|\mathcal P_{22}^n(s,u)|
	\le
	C_{\alpha,b',k}\sum_{k=0}^{n-1}\mathbb{I}_{s\in[t_k,t_{k+1})}u^{-\a-1}
|K^H(t_k,u)|\mathbb{I}_{u\leq t_k}.
	\qquad\qquad\quad\quad\quad\quad\quad
	\]
Similarly to $\mathcal{P}^n_{1}$, if we $s=u$, we have $s=u=t_k$. That is
$$\mathcal{P}^n_{22}(s,s)=0.$$	
For the term $\mathcal{P}^n_{3}(s,u)$, the change of variable $v=\frac{r-t_k}{s-t_k}$ gives that
	\begin{align}
	\int_{t_k}^{s}\frac{s^{-\a}-r^{-\a}}{(s-r)^{\a+1}}dr
	=(s-t_k)^{-\alpha}\int_{0}^{1}\frac{s^{-\a}-((s-t_k)v+t_k)^{-\a}}{(1-v)^{\a+1}}dv,\nonumber
\end{align}
Using \eqref{small ineq} again we have $$s^{-\a}-((s-t_k)v+t_k)^{-\a}\leq C_{\a}|(s-t_k)v+t_k|^{-\a-1}|(s-t_k)(1-v)|.$$
Substituting this estimate into the above integral representation yields
\begin{align}
	\mathcal{P}^n_{3}(s,u)&=\sum_{k=0}^{n-1}\frac{1}{\G(1-\a)}\mathbb{I}_{s\in[t_k,t_{k+1})}\a s^{\a}b'(\phi^n_{t_k})K^H(t_k,u)\mathbb{I}_{u\leq t_k}\int_{t_k}^{s}\frac{s^{-\a}-r^{-\a}}{(s-r)^{\a+1}}dr\nonumber\\&\leq C_{\a,b'}\sum_{k=0}^{n-1}\mathbb{I}_{s\in[t_k,t_{k+1})} s^{\a}K^H(t_k,u)\mathbb{I}_{u\leq t_k}(s-t_k)^{-\alpha}\int_{0}^{1}\frac{s^{-\a}-((s-t_k)v+t_k)^{-\a}}{(1-v)^{\a+1}}dv\nonumber\\&\leq C_{\a,b'}\sum_{k=0}^{n-1}\mathbb{I}_{s\in[t_k,t_{k+1})} s^{\a}K^H(t_k,u)\mathbb{I}_{u\leq t_k}(s-t_k)^{-\alpha}\int_{0}^{1}|(s-t_k)v+t_k|^{-\a-1}\frac{|(s-t_k)(1-v)|}{(1-v)^{\a+1}}dv\nonumber\\&\leq C_{\a,b'}\sum_{k=0}^{n-1}\mathbb{I}_{s\in[t_k,t_{k+1})} s^{\a}K^H(t_k,u)\mathbb{I}_{u\leq t_k}(s-t_k)^{1-\alpha}t_k^{-\a-1}\int_{0}^{1}(1-v)^{-\a}dv\nonumber\\&\leq C_{\a,b'}\sum_{k=0}^{n-1}\mathbb{I}_{s\in[t_k,t_{k+1})} s^{\a}K^H(t_k,u)\mathbb{I}_{u\leq t_k}.\nonumber 
\end{align}
Finally, if $s=u$, then $s\in[t_k,t_{k+1})$ and $u\leq t_k$ imply $s=u=t_k$, and hence
	\begin{align}
		\mathcal{P}^n_{3}(s,s)=0.\nonumber
	\end{align}
	As a consequence, we have
	\begin{align}
		Tr(\tilde{f}_n)&=\frac{1}{2}\int_{0}^{1}f_n(s,s)ds=\frac{1}{2}\int_{0}^{1}\mathcal{P}^n_1(s,s)+\mathcal{P}^n_2(s,s)+\mathcal{P}^n_3(s,s)ds\nonumber\\&=\frac{1}{2}\int_{0}^{1}\mathcal{P}^n_{1}(s,s)+\mathcal{P}^n_{21}(s,s)+\mathcal{P}^n_{22}(s,s)+\mathcal{P}^n_{3}(s,s)ds\nonumber\\&=0.\nonumber
	\end{align}
	To summarized what we have proved, Lemma \ref{Separation lemma} and Lemma \ref{norm} give us
	\begin{align}\label{J_{13}}
		\limsup _{\varepsilon \rightarrow 0}\ \E(\exp(\mathcal{I}^n_{12})|\|B^H\|<\varepsilon)=1.
	\end{align}
	Finally, it only remains to study the limit behavior of the term $\mathcal{I}^n_{13}$. For any $c\in \mathbb{R}$ and $\delta>0$ we can write
	\begin{align}
	\E&\Big(\exp(c\mathcal{I}^n_{13})|\|B^H\|\leq \varepsilon\Big)\nonumber\\
	&\leq e^{\delta}+\int_{\delta}^{\infty}e^{\xi}\mathbb{P}(|c\mathcal{I}^n_{13}|>\xi|\|B^H\|_{\b}\leq \varepsilon)d \xi +e^{\delta}\mathbb{P}(|c\mathcal{I}^n_{13}|>\delta |\|B^H\|_{\b}\leq \varepsilon).\nonumber
\end{align}
	Define the martingale $M^n_t=c\int_{0}^{t}d_H^{-1}s^{-\a}\Big(D^{\alpha}_{0+}u^{-\alpha}R_u^{n}\Big)(s)dW_s$.
	 To estimate whose quadratic variations we make use of the following expression of the residual term
	\begin{align}
		R_s^{n}&=b(\phi^{n}_{\eta_n(s)}+B_{\eta_n(s)}^H)-b(\phi^n_{\eta_n(s)})-b'(\phi^n_{\eta_n(s)})B_{\eta_n(s)}^H\nonumber\\&=\int_{0}^{1}\big(b'(\phi_{\eta_n(s)}^{n}+\lambda B_{\eta_n(s)}^H)-b'(\phi^n_{\eta_n(s)})\big)d\lambda B^H_{\eta_n(s)}\nonumber\\& =\int_{0}^{1}\int_{0}^{\lambda}b''(\phi_{\eta_n(s)}^{n}+\mu B_{\eta_n(s)}^H)d\mu d\lambda (B^H_{\eta_n(s)})^2\nonumber.
		\qquad\qquad\qquad\qquad\qquad\qquad
	\end{align}
	Using that $b''$ is Lipschitz with constant $L''$, we have
	\begin{align}
		&R_s^{n}-R_r^{n}\nonumber\\
		&=\int_{0}^{1}\int_{0}^{\lambda}\Big|b''(\phi^n_{\eta_n(s)}+\mu B_{\eta_n(s)}^H)-b''(\phi^n_{\eta_n(r)}+\mu B_{\eta_n(r)}^H)\Big|d\mu d\lambda (B^H_{\eta_n(s)})^2\qquad\quad\nonumber\\
		&\quad+\int_{0}^{1}\int_{0}^{\lambda}b''(\phi^n_{\eta_n(r)}+\mu B_{\eta_n(r)}^H)d\mu d\lambda \big|(B^H_{\eta_n(s)})^2-(B^H_{\eta_n(r)})^2\big|\nonumber\\
		&\leq\nonumber L''\Big(\frac{1}{2}|\phi^n_{\eta_n(s)}-\phi^n_{\eta_n(s)}|+\frac{1}{6}|B_{\eta_n(s)}-B_{\eta_n(r)}^2|\Big)(B^H_{\eta_n(s)})^2\nonumber\\
		&\quad+\frac{1}{2}\|b''\|_{\infty}\big|(B^H_{\eta_n(s)})^2-(B^H_{\eta_n(r)})^2\big|.\nonumber
	\end{align}
	Combining \eqref{weyl representation}, \eqref{integral transform} and \eqref{estimate of Remainder term}, we have
	\begin{align}
		&\G(1-\a)\Big|s^{\a}(D^{\a}_{0+}u^{-\a}R_u^{n})(s)\Big|\nonumber\\&=\Bigg|s^{-\a}R_s^{n}+\a s^{\a}\int_{0}^{s}\frac{s^{-\a}R_s^{n}-r^{-\a}R_r^{n}}{(s-r)^{\a+1}}dr\Bigg|\nonumber\\&\leq C_{\b}s^{-\a}\varepsilon^2+C_{\a,\b} s^{\a}\varepsilon^2\int_{0}^{s}\frac{|s^{-\a}-r^{-\a}|}{(s-r)^{\a+1}}dr+\a s^{\a}\int_{0}^{s}\frac{r^{-\a}|R_s^{n}-R_r^{n}|}{(s-r)^{\a+1}}dr\nonumber\\&\leq C_{\b}s^{-\a}\varepsilon^2+C_{\a,\b} s^{\a}\varepsilon^2\int_{0}^{s}\frac{|s^{-\a}-r^{-\a}|}{(s-r)^{\a+1}}dr+\frac{L''}{2}\a s^{\a}|B^H_{\eta_n(s)}|^2\int_{0}^{s}\frac{r^{-\a}|\phi^n_{\eta_n(s)}-\phi^n_{\eta_n(r)}|}{(s-r)^{\a+1}}dr\nonumber\\&\quad+\frac{L''}{6}\a s^{\a}|B^H_{\eta_n(s)}|^2\int_{0}^{s}\frac{r^{-\a}|B^H_{\eta_n(s)}-B^H_{\eta_n(r)}|}{(s-r)^{\a+1}}dr+\frac{\|b''\|_{\infty}}{2}\a s^{\a}\int_{0}^{s}\frac{r^{-\a}\big|(B^H_{\eta_n(s)})^2-(B_{\eta_n(r)}^H)^2\big|}{(s-r)^{\a+1}}dr\nonumber\\&\leq C_{\b}\varepsilon^2+C_{\a,\b} s^{-\a}\varepsilon^2+C_{\a,\b,n,L''}\varepsilon^2 s^{\a}|\eta_n(s)|^{2\b}\int_{0}^{s}\frac{r^{-\a}|\eta_n(s)-\eta_n(r)|^{H}}{(s-r)^{\a+1}}dr\nonumber\\&\quad+C_{\a,L''}\varepsilon^3s^{\a}|\eta_n(s)|^{2\b}\int_{0}^{s}\frac{r^{-\a}|\eta_n(s)-\eta_n(r)|^{\b}}{(s-r)^{\a+1}}dr\nonumber\\&\quad+C_{\a}\varepsilon^2 s^{\a}\int_{0}^{s}\frac{r^{-\a}|\eta_n(s)-\eta_n(r)|^{2\b}}{(s-r)^{\a+1}}dr+C_{\a}\varepsilon^2 s^{\a}\int_{0}^{s}\frac{r^{-\a}|\eta_n(s)-\eta_n(r)|^{\b}\eta_n(r)^{\b}}{(s-r)^{\a+1}}dr\nonumber\\&\leq C_{\b}\varepsilon^2+C_{\a,\b} s^{-\a}\varepsilon^2+C_{\a,\b,n,L''}\varepsilon^2 s^{2\b+1/2}+C_{\a,L''}\varepsilon^3s^{2\b+1/2}.\nonumber
	\end{align}
	As a consequence,
	\begin{align}
		\langle M^n \rangle_t=&c^2d_H^{-2}\int_{0}^{t}\Big(s^{\alpha}\Big(D^{\alpha}_{0+}u^{-\alpha}R_u^{n}\Big)(s)\Big)^2ds\leq C_{\a,\b,n,H,L''}\varepsilon^4.\nonumber
	\end{align}
	Applying the exponential inequality for martingales, we have
	\begin{align}\label{estimate of expon}
		\mathbb{P}\Big(\Big|c\int_{0}^{t}s^{-\a}\Big(I^{\alpha}_{0+}u^{\alpha}R_u^{n}\Big)(s)dW_s\Big|>\xi, \|B^H\|_{\b}\leq \varepsilon\Big)\leq \exp\Big(-\frac{\xi^2}{2C_{\a,\b,n,H,L''}\varepsilon^{4}}\Big),
	\end{align}
	for every real number $c$.\\
	Combining Lemma \ref{small ball of holder} and inequality \eqref{estimate of expon}, we see that
	\begin{align}
	\mathbb{P}&\Big(\Big|c\int_{0}^{1}s^{-\a}\Big(I^{\alpha}_{0+}u^{\alpha}R_u^{n}\Big)(s)dW_s\Big|>\xi \Big|\|B^H\|_{\b}\leq \varepsilon\Big)
	\leq \exp\Big(-\frac{\xi^2}{2C_{\a,\b,n,H,L''}\varepsilon^{4}}\Big)\exp\big(C_H \varepsilon^{-\frac{1}{H-\b}}\big).\nonumber
\end{align}
	\\
	Using the latter estimate we have for every $\delta>0$ and every $0<\varepsilon<1$
	\begin{align}
		\E&\Big(\exp(c\mathcal{I}^n_{13})|\|B^H\|_{\beta}\leq \varepsilon\Big)\nonumber\\
		&\leq e^{\delta}+\int_{\delta}^{\infty}\exp\Big\{\xi-\frac{\xi^2}{2C_{\a,\b,n,H,L''}\varepsilon^{4}}+C_H \varepsilon^{-\frac{1}{H-\b}}\Big\}d \xi \nonumber\\
		&\quad+\exp\Big\{\delta-\frac{\delta^2}{2C_{\a,\b,n,H,L''}\varepsilon^{4}}+C_H \varepsilon^{-\frac{1}{H-\b}}\Big\}.\nonumber
	\end{align}
	Letting $\varepsilon$ and then $\delta$ tend to zero, we obtain
	\begin{align}\label{J_{14}}
		\limsup _{\varepsilon \rightarrow 0}\ \E(\exp(c\mathcal{I}^n_{13})|\|B^H\|_{\b}<\varepsilon)\leq1,
	\end{align}
	for every real number $c$. \\
	
	In conclusion, the following expression is a consequence of Lemma \ref{Separation lemma} and inequalities \eqref{small all 2}, \eqref{J_2}, \eqref{J_3}, \eqref{J_4}, \eqref{the part of J_1}, \eqref{J_{11}}, \eqref{J_{13}} and \eqref{J_{14}}.
	$$
	\lim _{\varepsilon \rightarrow 0} \frac{\mathbb{P}(\|X^n-\phi^{n}\|_{\b}<\varepsilon)}{\mathbb{P}(\|B^H\|_{\b}<\varepsilon)}=\exp \left(-\frac{1}{2}\int_{0}^{1}\big|\dot{\phi}^{n}_s-d_H^{-1}s^{\a}\big(D^{\a}_{0+}u^{-\a}b(\phi^n_{\eta_n(u)})\big)(s)\big|^2ds\right).
	$$
	The proof of Theorem \ref{th:regular case} is complete.
\end{proof}
	\section{The most probable path of numerical solution process}\label{TMPP}
In this section, we derive Euler-Lagrange equations for the numerical Onsager-Machlup action functional of numerical solution process $X^n_t$. And then by numerical experiments to verify our theoretical results.

\textbf{Case 1.} 
When $\frac{1}{4}<H<\frac{1}{2}$, recall that the most probable path $\psi^n$ is the minimizer of the non-negative functional $J(\phi^n,\dot\phi^n)=-L^n(\phi^n,\dot\phi^n)$
with the numerical Onsager-Machlup action functional $L^n$ we have derived in Theorem \ref{th:singular case}.
Define $A^n_t = \dot\phi^n_t - d_H^{-1}t^{-\alpha} I^\alpha_{0+}
\big(u^\alpha b(\phi^n_{\eta_n(u)})\big)(t)$, 
then we have $J(\phi^n,\dot\phi^n)=\frac{1}{2}\int_0^1|A^n_t|^2 dt$.

For every $n$, let $v_t^n$ be a smooth test function with $v_0^n=v_1^n=0$ and perturb $\phi^n\to\phi^n+\kappa v^n$. 
The first variation is
\begin{align*}
	\delta J = \frac{d}{d\kappa}J(\phi^n+\kappa v^n)\Big|_{\kappa=0} 
	= \int_0^1 A^n_t \dot v_t^n dt
	-d_H^{-1}\int_0^1 A^n_t t^{-\alpha} I^\alpha_{0+}\big(u^\alpha b'(\phi^n_{\eta_n(u)}) v_{\eta_n(u)}^n\big)(t)dt.
\end{align*}
For the first term, classical integration by parts with $v_0^n=v_1^n=0$ gives
\begin{align*}
	\int_0^1 A^n_t \dot v_t^n dt = -\int_0^1 \dot A^n_t v_t^n dt .
\end{align*}
For the second term, we can use the fractional integration by parts formula
\begin{align*}
	\int_0^1 f(t)I^\alpha_{0+}g(t)dt = \int_0^1 g(t)I^\alpha_{1-}f(t)dt
	\quad\text{with}\quad
	I^\alpha_{1-}f(t)=\frac{1}{\Gamma(\alpha)}\int_t^1 (s-t)^{\alpha-1}f(s)ds,
\end{align*}
which yields
\begin{align*}
	\quad&\int_0^1 A^n_t t^{-\alpha} I^\alpha_{0+}\big(u^\alpha b'(\phi^n_{\eta_n(u)}) v_{\eta_n(u)}^n\big)(t)dt
	=\int_0^1 t^\alpha b'(\phi^n_{\eta_n(t)}) v_{\eta_n(t)}^n  I^\alpha_{1-}\big(s^{-\alpha}A^n_s\big)(t)dt\quad\quad\qquad\\
	&=\sum_{i=0}^{n-1} \int_{t_i}^{t_{i+1}} t^\alpha b'(\phi^n_{t_i}) v_{t_i} I^\alpha_{1-}\big(s^{-\alpha}A^n_s\big)(t)dt 
	=\sum_{i=1}^{n-1} v_{t_i}^n b'(\phi^n_{t_i}) \int_{t_i}^{t_{i+1}} u^\alpha I^\alpha_{1-}\big(s^{-\alpha}A^n_s\big)(u)du \\
	&=\int_0^1 v_t^n \Big(\sum_{i=1}^{n-1} \delta_{t_i}(t)b'(\phi^n_{t_i}) \int_{t_i}^{t_{i+1}} u^\alpha I^\alpha_{1-}\big(s^{-\alpha}A^n_s\big)(u)du\Big) dt.
\end{align*}

Combining both parts and setting $\delta J=0$ for all admissible $v_t^n$,
we can obtain that the most probable path $\psi^n$ satisfies  
\begin{align*}
	\int_0^1 v_t^n \big[-\dot A^n_t - d_H^{-1}\sum_{i=1}^{n-1} \delta_{t_i}(t) b'(\phi^n_{t_i}) \int_{t_i}^{t_{i+1}} u^\alpha I^\alpha_{1-}\big(s^{-\alpha}A^n_s\big)(u)du \big] dt = 0.\qquad\qquad\qquad\quad\quad
\end{align*}
By the calculus of variations, the Euler-Lagrange equation is
\begin{align}\label{EL1}
	\begin{cases}
		\dot A^n_t + \displaystyle\sum_{i=1}^{n-1} C^{n,i} \delta_{t_i}(t) = 0,\\[1.2em]
		A^n_t = \dot\phi^n_t - d_H^{-1}t^{-\alpha} I^\alpha_{0+}
		\big(u^\alpha b(\phi^n_{\eta_n(u)})\big)(t),
	\end{cases}
\end{align}
where $C^{n,i} = d_H^{-1}b'(\phi^n_{t_i}) \int_{t_i}^{t_{i+1}} u^\alpha I^\alpha_{1-}\big(s^{-\alpha}A^n_s\big)(u)du$.

Next, we will further analyze the above equation and obtain an explicit equation for $\psi^n$.
When $t \in (t_i, t_{i+1})$, we have $\dot A^n_t=0$, so there exist constants $A^{n,i}$ such that $A^n_t \equiv A^{n,i}$. The left fractional integral explicitly expands to
\begin{align*}
	&I^\alpha_{0+} \big(u^\alpha b(\phi^n_{\eta_n(u)})\big)(t) 
	= 
	\frac{1}{\Gamma(\alpha)} \left( \sum_{k=0}^{i-1} b(\phi^n_{t_k}) \int_{t_k}^{t_{k+1}} (t-u)^{\alpha-1} u^\alpha du + b(\phi^n_{t_i})\int_{t_i}^t (t-u)^{\alpha-1} u^\alpha du \right)\\
	= &
	\frac{t^{2\alpha}}{\Gamma(\alpha)} \left[
	\sum_{k=0}^{i-1} b(\phi^n_{t_k})
	\left(\!B_{\frac{t_{k+1}}{t}}\!(\alpha+1, \alpha)-B_{\frac{t_{k}}{t}}\!(\alpha+1, \alpha)\!\right)
	+b(\phi^n_{t_i})\left(\!B_{1}\!(\alpha+1, \alpha) - B_{\frac{t_i}{t}}\!(\alpha+1, \alpha) \!\right)
	\right],
\end{align*}
where $B_x(a, b) = \int_0^x y^{a-1}(1-y)^{b-1} dy$ means the incomplete Beta function with $a > 0$, $b > 0$ and $0 \leq x \leq 1$, then we have for all $t \in (t_i, t_{i+1})$
\begin{align*}
	\dot\phi^n_t 
	&=A^{n,i}+\frac{d_H^{-1}t^{\alpha}}{\Gamma(\alpha)} \left[
	\sum_{k=0}^{i-1} b(\phi^n_{t_k})\!
	\left(\!B_{\frac{t_{k+1}}{t}}\!(\alpha+1, \alpha)-B_{\frac{t_{k}}{t}}\!(\alpha+1, \alpha)\!\right)
	+b(\phi^n_{t_i})\!\left(\!B_{1}\!(\alpha+1, \alpha) - B_{\frac{t_i}{t}}\!(\alpha+1, \alpha) \!\right)
	\right].
\end{align*}
When $t=t_i$, we have the jump $C^n_i = A^n_{t_i^-} - A^n_{t_i^+}=A^{n,i-1} - A^{n,i}$.
By the definition of the right Riemann-Liouville fractional integral $I^\alpha_{1-}$ and Fubini's theorem, we obtain 
\begin{align*}
	C^{n,i} &= \frac{d_H^{-1}b'(\phi^n_{t_i})}{\Gamma(\alpha)} 
	\int_{t_i}^{t_{i+1}} u^\alpha \left( \int_u^1 (s-u)^{\alpha-1} s^{-\alpha} A^n_s ds \right) du\\
	&= \frac{d_H^{-1}b'(\phi^n_{t_i})}{\Gamma(\alpha)} 
	\int_{t_i}^{t_{i+1}} u^\alpha \left( \int_u^1 (s-u)^{\alpha-1} s^{-\alpha} 
	\sum_{k=0}^{n-1} A^{n,k} \mathbf{1}_{(t_k, t_{k+1})}(s)ds \right) du\\
	&= \frac{d_H^{-1}b'(\phi^n_{t_i})}{\Gamma(\alpha)} 
	\sum_{k=i}^{n-1}A^{n,k}
	\int_{t_i}^{t_{i+1}} u^\alpha \left( \int_{u\vee {t_k}}^{t_{k+1}} (s-u)^{\alpha-1} s^{-\alpha} ds \right) du.
\end{align*}
By the definition of $A^n_t$, it follows that the jump of $\dot{\phi}^n$ at the node $t_i$ is
$\dot{\phi}^n_{t_i^+} - \dot{\phi}^n_{t_i^-}=-C^{n,i}$.

Therefore, the Euler-Lagrange equation \eqref{EL1} can be rewritten as
\begin{align}\label{EL1exp}
	\begin{cases}
		\dot\phi^n_t 
		=A^{n,i}+\frac{d_H^{-1}t^{\alpha}}{\Gamma(\alpha)}G^{n,i}_t, &\quad t \in (t_i, t_{i+1}),\quad i=0,\dots,n-1,\\[1.2em]
		\dot{\phi}^n_{t_i^+} - \dot{\phi}^n_{t_i^-}=-C^{n,i}, &\quad i=1,\dots,n-1,
	\end{cases}
\end{align}
where 
$$
G^{n,i}_t=
\sum_{k=0}^{i-1} b(\phi^n_{t_k})\!
\left(\!B_{\frac{t_{k+1}}{t}}\!(\alpha+1, \alpha)-B_{\frac{t_{k}}{t}}\!(\alpha+1, \alpha)\!\right)
+b(\phi^n_{t_i})\!\left(\!B_{1}\!(\alpha+1, \alpha) - B_{\frac{t_i}{t}}\!(\alpha+1, \alpha) \!\right)
$$
and
$$
C^{n,i} =A^{n,i-1} - A^{n,i}= \frac{d_H^{-1}b'(\phi^n_{t_i})}{\Gamma(\alpha)} 
\sum_{k=i}^{n-1}A^{n,k}
\int_{t_i}^{t_{i+1}} u^\alpha \left( \int_{u\vee {t_k}}^{t_{k+1}} (s-u)^{\alpha-1} s^{-\alpha} ds \right) du.\qquad\quad\quad\quad\;\;\;
$$

\textbf{Case 2.} 
Similarly, according to Theorem \ref{th:regular case}, 
we can deduce the fractional Euler-Lagrange equation when $\frac{1}{2}<H<1$
\begin{align}\label{EL2}
	\begin{cases}
		\dot A^n_t + \displaystyle\sum_{i=1}^{n-1} C^{n,i} \delta_{t_i}(t) = 0,\\[1.2em]
		A^n_t = \dot\phi^n_t - d_H^{-1}t^\alpha D^\alpha_{0+}
		\big(u^{-\alpha} b(\phi^n_{\eta_n(u)})\big)(t),
	\end{cases}
\end{align}
where $C^{n,i} = d_H^{-1}b'(\phi^n_{t_i}) \int_{t_i}^{t_{i+1}} u^{-\alpha} D^\alpha_{1-}\big(s^{\alpha}A^n_s\big)(u)du$.

Next, we derive an explicit discrete form of \eqref{EL2}. 
For $t\in (t_i,t_{i+1})$, the equation $\dot A_t^n=0$ implies that there exist constants $A^{n,i}$ such that
\begin{align*}
	A_t^n\equiv A^{n,i},\qquad t\in (t_i,t_{i+1}),\quad i=0,\dots,n-1.
\end{align*}
For $t\in (t_i,t_{i+1})$, by the Weyl representation \eqref{weyl representation}, 
we can compute the left-sided Riemann-Liouville derivative 
\begin{align*}
	&\quad D^\alpha_{0+}\big(u^{-\alpha} b(\phi^n_{\eta_n(u)})\big)(t)
	=
	\frac{1}{\Gamma(1-\alpha)}
	\left(
	\frac{t^{-\alpha} b(\phi^n_{t_i})}{t^\alpha}
	+\alpha\int_0^t
	\frac{t^{-\alpha} b(\phi^n_{t_i})-u^{-\alpha} b(\phi^n_{\eta_n(u)})}{(t-u)^{\alpha+1}}du
	\right)\\
	&=
	\frac{1}{\Gamma(1-\alpha)}
	\Big[
	t^{-2\alpha}b(\phi^n_{t_i})
	+\alpha\sum_{k=0}^{i-1}\int_{t_k}^{t_{k+1}}
	\frac{t^{-\alpha} b(\phi^n_{t_i})-u^{-\alpha} b(\phi^n_{t_k})}{(t-u)^{\alpha+1}}du
	+\alpha b(\phi^n_{t_i})\int_{t_i}^t\frac{t^{-\alpha}-u^{-\alpha}}{(t-u)^{\alpha+1}}du
	\Big]\\
	&=\frac{1}{\Gamma(1-\alpha)}
	\Big[
	t^{-2\alpha}b(\phi^n_{t_i})
	+\alpha\sum_{k=0}^{i-1}\left(b(\phi^n_{t_i})-b(\phi^n_{t_k})\right)
	\int_{t_k}^{t_{k+1}}\!\!\frac{u^{-\alpha}}{(t-u)^{\alpha+1}}du
	+\alpha b(\phi^n_{t_i})
	\int_0^t\frac{t^{-\alpha}-u^{-\alpha}}{(t-u)^{\alpha+1}}du
	\Big]\\
	&=\frac{b(\phi^n_{t_i})}{\Gamma(1-\alpha)} 
	\Big[
	t^{-2\alpha}+\alpha \int_0^t\frac{t^{-\alpha}-u^{-\alpha}}{(t-u)^{\alpha+1}}du
	\Big]
	+ \frac{\alpha}{\Gamma(1-\alpha)} \sum_{k=0}^{i-1} \big(b(\phi^n_{t_i}) - b(\phi^n_{t_k})\big) \int_{t_k}^{t_{k+1}} \!\!\frac{u^{-\alpha}}{(t-u)^{\alpha+1}} du\\
	&=\frac{\Gamma(1-\alpha)}{\Gamma(1-2\alpha)} t^{-2\alpha} b(\phi^n_{t_i}) + \frac{\alpha}{\Gamma(1-\alpha)} \sum_{k=0}^{i-1} \big(b(\phi^n_{t_i}) - b(\phi^n_{t_k})\big) \int_{t_k}^{t_{k+1}}\!\! \frac{u^{-\alpha}}{(t-u)^{\alpha+1}} du,
\end{align*}
where the last identity follows from
$$ D^\alpha_{0+}(u^{-\alpha})(t) = \frac{1}{\Gamma(1-\alpha)}\Big[ t^{-2\alpha} + \alpha \int_0^t \frac{t^{-\alpha}-u^{-\alpha}}{(t-u)^{\alpha+1}}du \Big] = \frac{\Gamma(1-\alpha)}{\Gamma(1-2\alpha)}t^{-2\alpha}. $$
Indeed, by the definition of the left-sided Riemann-Liouville fractional derivative 
and the properties of Beta functions and Gamma functions, we have  
\begin{align*}
	D_{0+}^{\alpha}(t^{-\alpha})
	&=\frac{1}{\Gamma(1-\alpha)}\frac{d}{dt}\int_0^t (t-u)^{-\alpha}u^{-\alpha} du\nonumber\\&=\frac{1}{\Gamma(1-\alpha)}\frac{d}{dt}
	\int_0^1 (t-tv)^{-\alpha}(tv)^{-\alpha}t dv\\
	&=\frac{(1-2\alpha)t^{-2\alpha}}{\Gamma(1-\alpha)}\int_0^1 (1-v)^{-\alpha}v^{-\alpha} dv =\frac{(1-2\alpha)t^{-2\alpha}}{\Gamma(1-\alpha)}B(1-\alpha,1-\alpha)\\
	&=\frac{(1-2\alpha)t^{-2\alpha}}{\Gamma(1-\alpha)}\frac{\Gamma(1-\alpha)^2}{\Gamma(2-2\alpha)} 
	=(1-2\alpha)t^{-2\alpha}\frac{\Gamma(1-\alpha)}{\Gamma(2-2\alpha)}
	=\frac{\Gamma(1-\alpha)}{\Gamma(1-2\alpha)}t^{-2\alpha}.
\end{align*}
Therefore, for all $t\in (t_i,t_{i+1})$,
\begin{align*}
	\dot\phi_t^n
	=
	A^{n,i}
	+
	\frac{d_H^{-1}\Gamma(1-\alpha)}{\Gamma(1-2\alpha)} t^{-\alpha} b(\phi^n_{t_i}) + \frac{d_H^{-1}\alpha t^{\alpha}}{\Gamma(1-\alpha)} \sum_{k=0}^{i-1} \big(b(\phi^n_{t_i}) - b(\phi^n_{t_k})\big) \int_{t_k}^{t_{k+1}} \!\!\frac{u^{-\alpha}}{(t-u)^{\alpha+1}} du.
\end{align*}
When $t=t_i$, the jump of $A_t^n$ is
\begin{align*}
	C^n_i = A^n_{t_i^-} - A^n_{t_i^+}=A^{n,i-1} - A^{n,i},\qquad i=1,\dots,n-1.\qquad\quad
\end{align*}
On the other hand, by the Weyl representation for $D^\alpha_{1-}$ and Fubini's theorem, we obtain for $u\in (t_i,t_{i+1})$,
\begin{align*}
	C^n_i &= d_H^{-1}b'(\phi^n_{t_i}) \int_{t_i}^{t_{i+1}} u^{-\alpha} D^\alpha_{1-}\big(s^{\alpha}A^n_s\big)(u)du \\
	&= \frac{d_H^{-1}b'(\phi^n_{t_i})}{\Gamma(1-\alpha)} \int_{t_i}^{t_{i+1}}\Big[ 
	\frac{ A^{n,i}}{(1-u)^\alpha} 
	+ \alpha \int_u^1 \frac{ A^{n,i} - u^{-\alpha}s^{\alpha} A^n_s}{(s-u)^{\alpha+1}} ds 
	\Big]du \\
	&= \frac{d_H^{-1}b'(\phi^n_{t_i})}{\Gamma(1-\alpha)} \int_{t_i}^{t_{i+1}} \Big[ 
	\frac{A^{n,i}}{(1-u)^\alpha} 
	+ \alpha \sum_{k=i}^{n-1} \int_{t_k \vee u}^{t_{k+1}} 
	\frac{ A^{n,i} - u^{-\alpha}s^{\alpha} A^{n,k}}{(s-u)^{\alpha+1}} ds
	\Big]du.
\end{align*}
The jump condition applies to the residual $A^n$:
$A^n_{t_i^+}-A^n_{t_i^-}=-C^{n,i}$.
The fractional drift term may be singular at the nodes, so a finite jump of $\dot\phi^n$ is not asserted.

Consequently, the Euler-Lagrange equation \eqref{EL2} can be rewritten as
\begin{align}
	\label{EL2exp}
	\begin{cases}
		\displaystyle
		\dot\phi_t^n
		=
		A^{n,i}+G^{n,i}_t,
		&\quad t\in (t_i,t_{i+1}),\ \ i=0,\dots,n-1,\\[1.5em]
		A^n_{t_i^+}-A^n_{t_i^-}=-C^{n,i},
		&\quad i=1,\dots,n-1,
	\end{cases}
\end{align}
where
\begin{align*}
	G^{n,i}_t=
	\frac{d_H^{-1}\Gamma(1-\alpha)}{\Gamma(1-2\alpha)} t^{-\alpha} b(\phi^n_{t_i}) + \frac{d_H^{-1}\alpha t^{\alpha}}{\Gamma(1-\alpha)} \sum_{k=0}^{i-1} \big(b(\phi^n_{t_i}) - b(\phi^n_{t_k})\big) \int_{t_k}^{t_{k+1}}\!\! \frac{u^{-\alpha}}{(t-u)^{\alpha+1}} du.
\end{align*}
and
\begin{align*}
	C^{n,i}
	=
	\frac{d_H^{-1}b'(\phi^n_{t_i})}{\Gamma(1-\alpha)}
	\int_{t_i}^{t_{i+1}}
	\Big[
	\frac{A^{n,i}}{(1-u)^\alpha}
	+\alpha\sum_{k=i}^{n-1}
	\int_{u\vee t_k}^{t_{k+1}}
	\frac{A^{n,i}-u^{-\alpha} s^{\alpha}A^{n,k}}{(s-u)^{\alpha+1}}ds
	\Big]du.\qquad\quad
\end{align*}

\section{Examples}\label{examples}
In this section, we present an example to illustrate our theoretical result. 
We consider the periodic potential
$$
V(x)=\frac{\cos(\pi x)}{\pi},
$$
and the corresponding drift function
$$
b(x)=-V'(x)=\sin(\pi x).
$$
The potential \(V\) has local minima at
$x=2m+1$, $m\in\mathbb Z$, 
and local maxima at 
$x=2m$, $m\in\mathbb Z$, 
representing metastable states of the underlying dynamical system. 
The stochastic dynamics may escape from one potential well and reach the other by overcoming the energy barrier located between the two metastable states. 
	Consider the following Euler scheme for a scalar SDE driven by a fractional Brownian motion with Hurst index $H \in (\frac{1}{4}, \frac{1}{2})$
	\begin{equation}\label{eq:example1}
		X^n_t =-1+\int_0^t
		\sin\!\left(\pi X^n_{\eta_n(s)}\right)ds
		+B^H_t,
		\qquad t\in[0,1].
	\end{equation}
	Here, we consider a uniform partition of the interval $[0,1]$ 
	defined by $t_i = \frac{i}{n}$ for $i = 0, 1, \dots, n$. 
	Furthermore, for any integer $n \ge 1$, 
	the function $\eta_n(t)$ is defined as $\eta_n(t) = t_k$ for $t \in [t_k, t_{k+1})$. 
	
	According to Theorem \ref{th:singular case}, the corresponding Onsager-Machlup action functional is given by
	\begin{equation*}
	L^n(\phi^n,\dot{\phi}^n)
	=-\frac12\int_0^1
	\left|
	\dot\phi^n_s
	-d_H^{-1}s^{-\alpha}
	\left(
	I^\alpha_{0+}
	\left[
	u^\alpha\sin\!\left(\pi\phi^n_{\eta_n(u)}\right)
	\right]
	\right)(s)
	\right|^2ds.
	\end{equation*}
	and by Equ. \eqref{EL1exp}, the associated fractional Euler-Lagrange equation is
	\begin{align*}
		\begin{cases}
			\displaystyle
			\dot\phi^n_t
			=
			A^{n,i}
			+\frac{d_H^{-1}t^\alpha}{\Gamma(\alpha)}G^{n,i}_t,
			&
			t\in(t_i,t_{i+1}),
			\quad i=0,\ldots,n-1,
			\\[1.2em]
			\displaystyle
			\dot\phi^n_{t_i^+}-\dot\phi^n_{t_i^-}
			=
			-C^{n,i},
			&
			i=1,\ldots,n-1,
		\end{cases}
	\end{align*}
	where
	\begin{align*}
		G^{n,i}_t
		&=\sum_{k=0}^{i-1}
		\sin\!\left(\pi\phi^n_{t_k}\right)
		\left[
		B_{\frac{t_{k+1}}{t}}(\alpha+1,\alpha)
		-B_{\frac{t_k}{t}}(\alpha+1,\alpha)
		\right]\\
		&\quad
		+\sin\!\left(\pi\phi^n_{t_i}\right)
		\left[
		B_1(\alpha+1,\alpha)
		-B_{\frac{t_i}{t}}(\alpha+1,\alpha)
		\right],
	\end{align*}
	and
	\begin{align*}
		C^{n,i}
		=A^{n,i-1}-A^{n,i}
		=\frac{\pi\cos\!\left(\pi\phi^n_{t_i}\right)}{\Gamma(\alpha)}
		\sum_{k=i}^{n-1}A^{n,k}
		\int_{t_i}^{t_{i+1}}u^\alpha
		\left(
		\int_{u\vee t_k}^{t_{k+1}}
		(s-u)^{\alpha-1}s^{-\alpha}ds
		\right)du.
	\end{align*}
		\begin{figure}[htbp]
		\centering
		\includegraphics[width=1.0\textwidth]{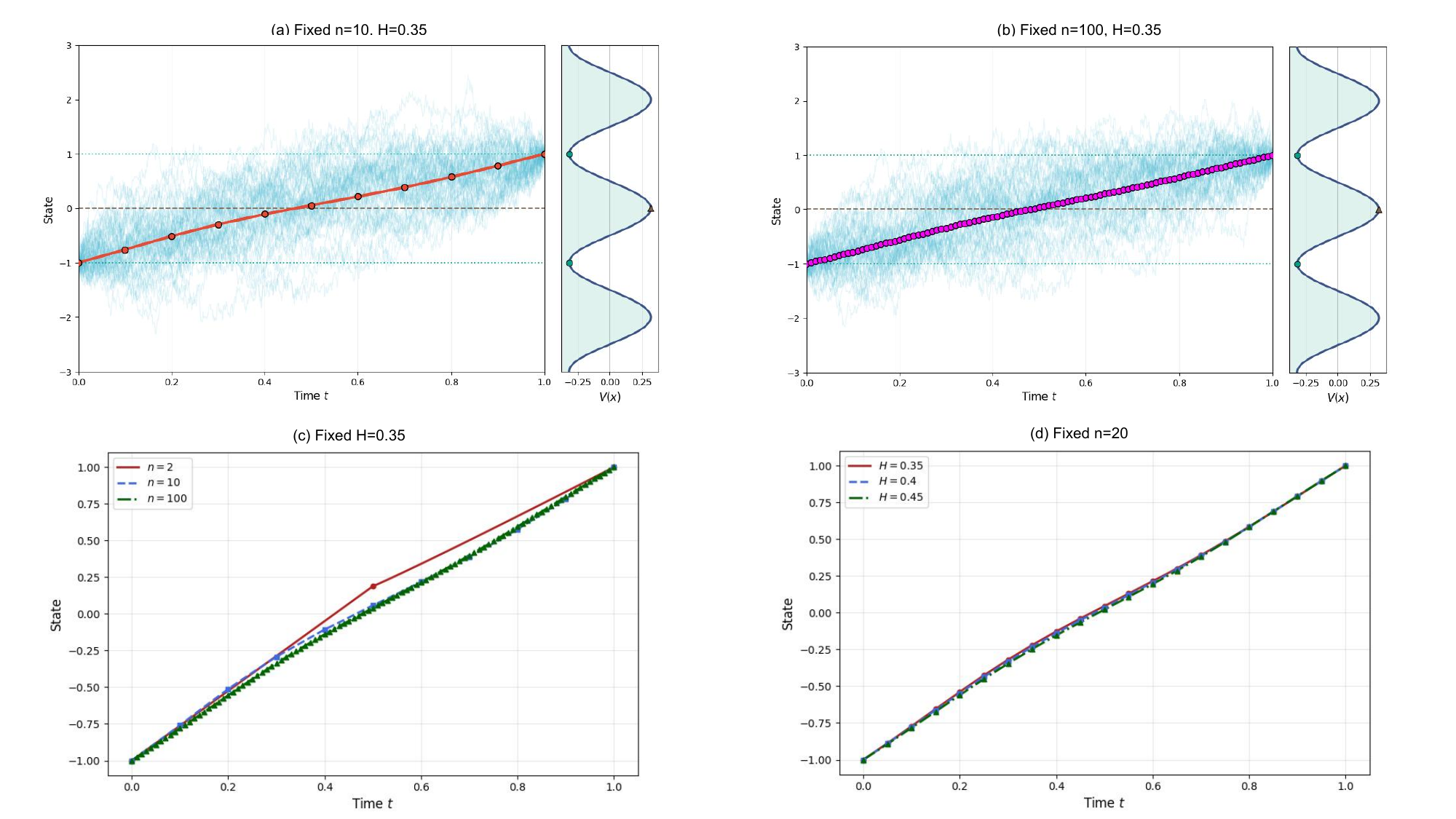}
		\caption{(a) The pattern of diferent sample paths (white lines) of the system \eqref{eq:example1} and the most probable path (magenta line) for $n=10$ and $H=0.35$.
			(b) The pattern of diferent sample paths (white lines) of the system \eqref{eq:example1} and the most probable path (orange line) for $n=100$ and $H=0.35$.
			(c) The change of the most probable paths with respect to
			different discretization parameters $n=2, 10, 100$ for fixed $H=0.35$.
			(d) The change of the most probable paths with respect to
			different Hurst parameters $H=0.35, 0.4, 0.45$ for fixed $n=20$.
		}
		\label{fig}
	\end{figure}
	The numerical results are presented below to illustrate the behavior of the
	most probable paths in the double-well potential landscape
	with different $n$ and the Hurst parameter $H$ in Figure \ref{fig}. 
	The potential barrier located at $x=0$ separates the two metastable states,
	and the transition path describes the most likely mechanism for overcoming
	this barrier under fractional stochastic perturbations.
	In Figures \ref{fig}(a)--(b), we demonstrate the relationship between the stochastic trajectories generated by the SDE \eqref{eq:example1} and the corresponding most probable path. We fix the Hurst parameter as 
	$H=0.35$ and investigate the influence of the time discretization parameter 
	$n=10$ and $n=100$. 
    It can be observed that the numerical paths become smoother and more 
	stable as the discretization becomes finer. 
	Figure \ref{fig}(c) further compares the most probable paths obtained with different values of $n=2, 10, 100$ with fixed $H=0.35$.
	In Figure \ref{fig}(d), we fix the discretization parameter at $n=20$ 
	and vary the Hurst parameter $H=0.35, 0.4, 0.45$. 

\appendix
\section{Appendix}\label{Appendix}
\subsection{Convergence of Euler scheme}
Since $b\in C_b^2$, by the elementary inequality
$|u+v+w|^p\le 3^{p-1}(|u|^p+|v|^p+|w|^p)$, we easily get
\begin{align}
	\EE |X_t|^p\leq C_p|x|^p+C_p\int_{0}^{1}\EE |b(X_t)|^pds+\EE|B_t^H|^p\leq C_{p,H}<\infty\nonumber
\end{align}
Simnilarly, we can also get $\EE |X_t^n|^p<\infty$. 
 For $s\in[0,1)$, the definition of the Euler scheme gives that
 \[
 X_s^n-X_{\eta_n(s)}^n
 =
 (s-\eta_n(s))b(X_{\eta_n(s)}^n)
 +B_s^H-B_{\eta_n(s)}^H.
 \]
 Consequently,
 \begin{align}
 	\mathbb E|X_s^n-X_{\eta_n(s)}^n|^p
 	&\le
 	C_p n^{-p}
 	+C_p\mathbb E|B_s^H-B_{\eta_n(s)}^H|^p
 	\nonumber\\
 	&\le C_{p}\bigl(n^{-p}+n^{-pH}\bigr)
 	\le C_p n^{-pH}.
 	\label{eq:euler-increment}
 \end{align}
 
 By Holder's inequality, Tonelli's theorem, and the
 Lipschitz continuity of $b$, we have that
 \begin{align}
 	\mathbb E|X_t-X_t^n|^p
 	&\le
 	t^{p-1}\int_0^t
 	\mathbb E|b(X_s)-b(X_{\eta_n(s)}^n)|^p\,ds
 	\nonumber\\
 	&\le
 	C_{p,L}\int_0^t\mathbb E|X_s-X_s^n|^p\,ds
 	\nonumber\\
 	&\quad+
 	C_{p,L}\int_0^t
 	\mathbb E|X_s^n-X_{\eta_n(s)}^n|^p\,ds
 	\nonumber\\
 	&\le
 	C_{p,L}\int_0^t\mathbb E|X_s-X_s^n|^p\,ds
 	+C_{p,L}n^{-pH},
 	\label{eq:euler-error}
 \end{align}
 where the last step follows from
 \eqref{eq:euler-increment}.
 Gronwall's inequality gives that
 \[
 \sup_{t\in[0,1]}\mathbb E|X_t-X_t^n|^p
 \le C_{p,L}n^{-pH},
 \]
 the constant $C_{p,L}$ is independent of $n$. So the proof of convergence of Euler scheme is complete.
\subsection{Approximate limits in Wiener space}
Before introducing the following theorem, we first recall some definitions and results on approximate limits in the Wiener space concerning measurable semi-norms for exponentials of random variables in the first and second Wiener chaos.

Let $W=\{W_t,t\in[0,1]\}$ be a Wiener process defined in the canonical probability space $(\Omega,\mathcal{F},\mathbb{P})$. $\Omega$ is the space of continuous functions vanishing
at zero, and $\mathbb{P}$ is the Wiener space. Let $H^1$ be the Cameron-Martin space, that is, the space of all absolutely continuous functions $h:[0,1]\to\mathbb{R}$ such that $h'\in H=L^2([0,1],\mathbb{R})$. The scalar product in $H^1$ is defined by
$$\langle h,g \rangle_{H^1}=\langle h',g' \rangle_{H},$$
for all $h,g\in H^1.$

Let $Q:H^1\to H^1$ be an orthogonal projection such that $\dim QH^1<\infty$. $Q$ can be written as
$$Qh=\sum_{i=1}^{n}\langle h, h_i\rangle h_i,$$
where $(h_1,\cdots, h_n)$ is an orthonormal sequence in $QH^1$. We can also define the $H^1$-valued random variable
$$Q^W=\sum_{i=1}^{n}\Big(\int_{0}^{1}h_i'(s)dW_s\Big)h_i.$$
Note that $Q^W$ does not depend on $(h_1,\cdots, h_n).$

A sequence of orthogonal projections $Q_n$ on $H^1$ is called the approximating sequence of projections if $\dim Q_nH^1<\infty$ and $Q_n$ converges strongly increasing to the identity operator in $H^1$.
\begin{defn}[\cite{MN02} Definition $1$]
	A semi-norm $N$ on $H^1$ is called a measurable semi-norm if there exists a random variable $\tilde{N}<\infty$ a.s, such that for all approximating sequence of projections $Q^n$ on $H^1$, the sequence $N(Q^W_n)$ converges in probability to $\tilde{N}$ and $\mathbb{P}(\tilde{N}\leq\varepsilon)> 0$ for all $\varepsilon>0$. If moreover $N$ is a norm on $H^1$, then is called measurable norm.
\end{defn}
We will make use of the following result on measurable semi-norms.
\begin{lem}[\cite{MN02} Lemma $1$]
	Let $N_n$ be a nondecreasing sequence of measurable semi-norms. Suppose that $\tilde{N}:= \mathbb{P}-\lim_{n\to\infty}\tilde{N}_n$ exists and $\mathbb{P} (\tilde{N}\geq \varepsilon) > 0$ for all $\varepsilon>0$. Then $N = \lim_{n\to\infty}N_n$ is a measurable semi-norm if this limit exists on $H^1$.
\end{lem}
The next two lemmas will be basic ingredients in the computation of the
numerical Onsager-Machlup action functional. Their proofs can be found in \cite{Har02}.
\begin{lem}[\cite{Har02}\label{no random function} Theorem $6$]
	Let $N$ be a measurable norm on $H^1$. Then,
	$$\lim_{\varepsilon\to 0}\E\Big(\exp\Big(\int_{0}^{1}h(s)dW_s\Big)|\tilde{N}<\varepsilon\Big)=1,$$
	for all $h\in L^2([0,1]).$
\end{lem}
We recall that an operator $K:H\to H$ is nuclear iff
$$\sum_{n=1}^{\infty}|\langle Ke_n,g_n\rangle|<\infty,$$
for all $B=(e_n)_n, B'=(g_n)_n$ orthonormal sequences in $H$. We define the trace of a nuclear operator $K$ by
$$Tr K=\sum_{i=1}^{\infty}\langle Ke_n,e_n\rangle,$$
for any $B=(e_n)_n$ orthonormal sequence in $H$. The definition is independent of the sequence we have chosen. Given a symmetric function $f\in L^2([0, 1]^2),$ the Hilbert-Schmidt operator $K(f): H \to H$ associated with $f$, defined by
\begin{align}
	(K(f))(h)(t)=\int_{0}^{t}f(t,u)h(u)du,\nonumber
\end{align}
is nuclear iff $\sum_{n=1}^{\infty}|\langle Ke_n,g_n\rangle|<\infty$ for all $B=(e_n)_n$ orthonormal sequence in $H$. If $f$ is continuous and the operator $K(f)$ is nuclear, we can compute its trace as follows:
$$Tr(f):=Tr K(f)=\int_{0}^{1}f(s,s)ds.$$
\begin{lem}[\cite{Har02}\label{key lemma} Theorem $8$]
	Let $f$ be a symmetric function in $L^2([0,1]^2)$ and let $N$ be a measurable norm. If $K(f)$ is nuclear, then
	\begin{align}
		\lim_{\varepsilon\to 0} E \Big(\exp\Big(\int_{0}^{1}\int_{0}^{1}f(s,t)dW_s dW_t\Big)|\tilde{N}<\varepsilon\Big)= e^{-Tr(f)}.\nonumber
	\end{align}
\end{lem}
\begin{lem}[\cite{LL98} \text{Theorem $1.1$}]
	Let $B^H$ be a fractional Brownian motion. Then
	$$\lim_{\varepsilon\to 0}\varepsilon^{\frac{1}{H}}\log \mathbb{P}\Big(\sup\limits_{t \in[0, 1]}\big|B^H\big| \leq \varepsilon\Big)=-C_H,$$
	where $C_H$ is a positive constant.
\end{lem}
\begin{lem}[\cite{KLS95} \text{Lemma $3.1$}]\label{small ball of holder}
	Let $B^H$ be a fractional Brownian motion and let $0\leq \beta<H$. Then there exists constants $0<K_1\leq K_2<\infty$ depending on $H$ and $\beta$ such that for all $0<\varepsilon<1$
	$$-K_2\varepsilon^{-\frac{1}{H-\beta}}\leq \log \mathbb{P}\Big(\sup\limits_{0\leq s,t\leq 1}\big|\frac{|B^H_t-B^H_s|}{|t-s|^{\beta}}\big| \leq \varepsilon\Big)\leq -K_1\varepsilon^{-\frac{1}{H-\beta}}.$$
\end{lem}
Then, we define the following norms on $H^1$:
$$N_H(h)=\sup\limits_{t \in[0, 1]}\big|\int_{0}^{t}K^H(t,s)h'(s)ds\big|,$$
$$N_{H,\beta}(h)=\sup\limits_{t,r\in [0, 1]}\frac{\Big|\int_{0}^{t}K^H(t,s)h'(s)ds-\int_{0}^{r}K^H(t,s)h'(s)ds\Big|}{|t-r|^\beta},$$
for $0<\beta<H.$
\begin{lem}[\cite{MN02} \text{Lemma 6}]\label{norm}
	$N_H$ and $N_{H,\beta}$ with $0<\beta\leq H$ are measurable norms and we have $\tilde{N}_H=\|B^H\|_{\infty}$ and $\tilde{N}_{H,\beta}=\|B^H\|_{\beta}$.
\end{lem}
\subsection{Technical lemmas}
\begin{lem}[\cite{IW14} \text{pp $536$-$537$}]\label{Separation lemma}
	For a fixed $n\geqslant1$, let $I_1, \ldots, I_n $ be $n$ random variables defined on $(\Omega,\mathcal{F},\mathbb{P})$ and $\{A_\varepsilon;\varepsilon >0\}$ a family of sets in $\mathcal{F}$. Suppose that for any $c\in \mathbb{R}$ and any $i=1,\dots, n$, if we have
	\begin{align}
		\limsup _{\varepsilon \rightarrow 0} \E\Big(\exp(cI_i)\vert A_\varepsilon\Big)\leqslant 1.\nonumber
	\end{align}
	Then
	\begin{align}
		\lim_{\varepsilon\to 0} \E \left [\exp \left(\sum_{i=1}^{n}cI_i \right ) \bigg| A_\varepsilon \right ]= 1.\nonumber
	\end{align}
\end{lem}
\begin{lem}\label{girsanov}
	Let $\eta$ be the process defined by \eqref{eta}. Then $\eta$ is adapted and
	$$\E\Big(\exp\Big(\int_{0}^{1}\eta_s dW_s-\frac{1}{2}\int_{0}^{1}\eta_s^2ds\Big)=1.$$
\end{lem}
\begin{proof}
	The proof of this lemma is similar to Lemma $10$ of \cite{MN02}. So the proof will be omitted.
\end{proof}
\begin{lem}\label{singular reace class}
	Assume $\frac{1}{4}<H<\frac{1}{2}$. Let $f_n$ be the function defined by
	\begin{align}
		f_n(s,r)&=\sum_{k=0}^{n-1}\frac{1}{\Gamma(\a)}s^{-\a}\mathbb{I}_{s\in[t_k,t_{k+1})}\sum_{j=0}^{k-1}K^H(t_j,r)b'(\phi^n_{t_j})\int_{t_j}^{t_{j+1}}u^{\a}(s-u)^{\a-1}du\mathbb{I}_{r\leq t_j}\nonumber\\&\quad+\sum_{k=0}^{n-1}\frac{1}{\Gamma(\a)}s^{-\a}b'(\phi^n_{t_k})K^H(t_k,r)\mathbb{I}_{s\in[t_k,t_{k+1})}\int_{t_k}^{s}u^{\a}(s-u)^{\a-1}du\mathbb{I}_{r\leq t_k},\nonumber
	\end{align}
	where $s,r\in(0,1]$ and $\phi^n$ is such that $\phi^n-x\in\mathcal{H}^p$ for $p>\frac{1}{H}$. Let $\tilde{f}_n$ be the
	symmetrization of $f_n$. Then the operator $K(\tilde{f}_n)$ defined by $K( \tilde{f}_n)(h)(s)=\int_{0}^{s}\tilde{f}_n(s, r)h(r)dr$ is nuclear.
\end{lem}
\begin{proof}
	We decompose
	\begin{align}
		f_n(s,r)&:=\mathcal{M}^n_1(s,r)+\mathcal{M}^n_2(s,r),\nonumber
	\end{align}
	where $\mathcal{M}^n_1$ and $\mathcal{M}^n_1$ denote the first and the second terms,
	respectively.
	
	For the first term $\mathcal{M}^n_1$, let $h\in L^2([0,1])$. Then we have
	\begin{align*}
		K(\mathcal{M}^n_1)(h)(s)
		&=
		\sum_{k=0}^{n-1}
		\frac{s^{-\alpha}}{\Gamma(\alpha)}
		\mathbb I_{s\in[t_k,t_{k+1})}
		\sum_{j=0}^{k-1}
		b'(\phi^n_{t_j})
		\\
		&\quad\cdot
		\left(
		\int_0^{t_j}K^H(t_j,r)h(r)dr
		\right)
		\int_{t_j}^{t_{j+1}}
		u^\alpha(s-u)^{\alpha-1}du .
	\end{align*}
	For each fixed pair $(k,j)$ with $0\leq j\leq k-1$, define
	\[
	g_{k,j}(s)
	=
	\frac{s^{-\alpha}}{\Gamma(\alpha)}
	\mathbb I_{s\in[t_k,t_{k+1})}
	\int_{t_j}^{t_{j+1}}
	u^\alpha(s-u)^{\alpha-1}du ,
	\]
	and
	\[
	L_{k,j}(h)
	=
	b'(\phi^n_{t_j})
	\int_0^{t_j}K^H(t_j,r)h(r)dr .
	\]
	Then
	\[
	K(\mathcal{M}^n_1)(h)(s)
	=
	\sum_{k=0}^{n-1}
	\sum_{j=0}^{k-1}
	g_{k,j}(s)L_{k,j}(h).
	\]
	Since the number of pairs $(k,j)$ is finite, we have
	\[
	\operatorname{rank}(K(\mathcal{M}^n_1))
	\leq
	\frac{n(n-1)}2 .
	\]
	Hence $K(f_n^{(1)})$ is a finite-rank operator.
	
	For \(j\leq k-1\) and \(s\in[t_k,t_{k+1})\), we have
	\(t_{j+1}\leq t_k\leq s\). Since \(u\in[t_j,t_{j+1}]\) which implies that $u^\alpha\leq s^\alpha$, we have 
	\begin{align}
		\int_{t_j}^{t_{j+1}}
		u^\alpha(s-u)^{\alpha-1}\,du
		\leq
		s^\alpha\int_{t_j}^{s}(s-u)^{\alpha-1}\,du=
		\frac{s^\alpha}{\alpha}(s-t_j)^\alpha.\nonumber
	\end{align}
	Consequently,
	\begin{align}
		|g_{k,j}(s)|
		\leq
		\frac{1}{\alpha\Gamma(\alpha)}
		\mathbb{I}_{[t_k,t_{k+1})}(s),\nonumber
	\end{align}
	where we have used \(0\leq s-t_j\leq 1\). It follows that
	\begin{align}
		\|g_{k,j}\|_{L^2([0,1])}^2
		\leq
		\frac{1}
		{\alpha^2\Gamma(\alpha)^2}
		\int_{t_k}^{t_{k+1}}1\,ds=
		\frac{1}
		{\alpha^2\Gamma(\alpha)^2}
		(t_{k+1}-t_k)
		<\infty.\nonumber
	\end{align}
	Thus,
	\[
	g_{k,j}\in L^2([0,1]).
	\]
	
	Moreover, for \(h\in L^2([0,1])\), by the Cauchy--Schwarz inequality, we have
	\begin{align}
		|L_{k,j}(h)|
		&\leq
		\|b'\|_{\infty}\left(
		\int_0^{t_j}|K^H(t_j,r)|^2\,dr
		\right)^{1/2}
		\left(
		\int_0^{t_j}|h(r)|^2\,dr
		\right)^{1/2} \nonumber\\
		&\leq
		\|b'\|_{\infty}\|K^H(t_j,\cdot)\|_{L^2([0,t_j])}
		\|h\|_{L^2([0,1])}.\nonumber
	\end{align}
	Since \(K^H(t_j,\cdot)\in L^2([0,t_j])\) and $b'$ is bounded, we have  \(L_{k,j}\) is a bounded linear functional on \(L^2([0,1])\). Hence, for each \(k,j\),
	\begin{equation}
		(g_{k,j}\otimes L_{k,j})(h)=L_{k,j}(h)g_{k,j}\nonumber
	\end{equation}
	is a rank-one nuclear operator. Since
	\begin{equation}
		K(\mathcal{M}^n_1)
		=
		\sum_{k=0}^{n-1}\sum_{j=0}^{k-1}
		g_{k,j}\otimes L_{k,j}
	\end{equation}
	is a finite sum of rank-one operators, \(K(\mathcal{M}^n_1)\) is nuclear.
	
	For the second term $\mathcal{M}^n_2$, we similarly obtain
	\begin{align*}
		K(\mathcal{M}^n_2)(h)(s)
		&=
		\sum_{k=0}^{n-1}
		\frac{s^{-\alpha}}{\Gamma(\alpha)}
		\mathbb I_{s\in[t_k,t_{k+1})}
		b'(\phi^n_{t_k})
		\\
		&\quad\times
		\left(
		\int_0^{t_k}K^H(t_k,r)h(r)dr
		\right)
		\int_{t_k}^{s}
		u^\alpha(s-u)^{\alpha-1}du .
	\end{align*}
	
	Define
	\[
	g_k(s)
	=
	\frac{s^{-\alpha}}{\Gamma(\alpha)}
	\mathbb I_{s\in[t_k,t_{k+1})}
	\int_{t_k}^{s}
	u^\alpha(s-u)^{\alpha-1}du ,
	\]
	and
	\[
	L_k(h)
	=
	b'(\phi^n_{t_k})
	\int_0^{t_k}K^H(t_k,r)h(r)dr .
	\]
	Then
	\[
	(K(f_n^{(2)})h)(s)
	=
	\sum_{k=0}^{n-1}g_k(s)L_k(h).
	\]
	Therefore,
	\[
	\operatorname{rank}(K(f_n^{(2)}))\leq n,
	\]
	and $K(f_n^{(2)})$ is finite-rank. By the similar argument as $\mathcal{M}^n_1$, we have $K(\mathcal{M}^n_2)$ is nuclear.
	
	Consequently,
	\[
	K(f_n)
	=
	K(\mathcal{M}^n_2)+K(\mathcal{M}^n_2)
	\]
	is a finite-rank nuclear operator.
	
	So the symmetrization of $f_n$ is also finite-rank nuclear operator.
\end{proof}

\begin{lem}\label{regular trace class}
	Assume $H\geq\frac{1}{2}$. Let $f_n$ be the function defined by
	\begin{align}
		f_n(s,u)&=\sum_{k=0}^{n-1}\frac{1}{\G(1-\a)}\mathbb{I}_{s\in[t_k,t_{k+1})}\Bigg[s^{-\a}b'(\phi^n_{t_k})K^H(t_k,u)\mathbb{I}_{u\leq t_k}+\a s^{\a}\nonumber\\&\quad\cdot\sum_{j=0}^{k-1}\int_{t_j}^{t_{j+1}}\frac{s^{-\a}b'(\phi^n_{t_k})K^H(t_k,u)\mathbb{I}_{u\leq t_k}-r^{-\a}b'(\phi^n_{t_j})K^H(t_j,u)\mathbb{I}_{u\leq t_j}}{(s-r)^{\a+1}}dr\nonumber\\&\quad+\a s^{\a}b'(\phi^n_{t_k})K^H(t_k,u)\mathbb{I}_{u\leq t_k}\int_{t_k}^{s}\frac{s^{-\a}-r^{-\a}}{(s-r)^{\a+1}}dr\Bigg].\nonumber
	\end{align}
	where $s,u\in(0,1]$ and $\phi^n$ is such that $\phi^n-x\in\mathcal{H}^p$ for $p>\frac{1}{H}$. Let $\tilde{f}_n$ be the
	symmetrization of $f_n$. Then the operator $K(\tilde{f}_n)$ defined by $K( \tilde{f}_n)(h)(s)=\int_{0}^{s}\tilde{f}_n(s, u)h(u)du$ is nuclear.
\end{lem}
\begin{proof}
	We firstly can rewrite $f_n$ as
	\begin{align}
		f_n(s,u)&=\sum_{k=0}^{n-1}\frac{1}{\G(1-\a)}\mathbb{I}_{s\in[t_k,t_{k+1})}s^{-\a}b'(\phi^n_{t_k})K^H(t_k,u)\mathbb{I}_{u\leq t_k}+\a s^{\a}\nonumber\\&\quad\cdot\sum_{k=0}^{n-1}\frac{1}{\G(1-\a)}\mathbb{I}_{s\in[t_k,t_{k+1})}\sum_{j=0}^{k-1}\int_{t_j}^{t_{j+1}}\frac{s^{-\a}b'(\phi^n_{t_k})K^H(t_k,u)\mathbb{I}_{u\leq t_k}-r^{-\a}b'(\phi^n_{t_j})K^H(t_j,u)\mathbb{I}_{u\leq t_j}}{(s-r)^{\a+1}}dr\nonumber\\&\quad+\a s^{\a}\sum_{k=0}^{n-1}\frac{1}{\G(1-\a)}\mathbb{I}_{s\in[t_k,t_{k+1})}b'(\phi^n_{t_k})K^H(t_k,u)\mathbb{I}_{u\leq t_k}\int_{t_k}^{s}\frac{s^{-\a}-r^{-\a}}{(s-r)^{\a+1}}dr.\nonumber\\&:=f_{n,1}(s,u)+f_{n,2}(s,u)+f_{n,3}(s,u).\nonumber
	\end{align}
	For $k=0,\ldots,n-1$, define
	\[
	\Psi_k(u):=K^H(t_k,u)\mathbf 1_{\{u\le t_k\}}.
	\]
	We now show that $f_n$ admits a finite-dimensional decomposition. From the definition of $f_n$, for every $s\in[t_k,t_{k+1})$, all the $u$-dependence appears only through
	\[
	\Psi_0(u),\ldots,\Psi_k(u).
	\]
	Indeed, the terms $f_{n,1}$ and $f_{n,3}$ contain only $\Psi_k$, while
	$f_{n,2}$ contains only $\Psi_j$ with $j=0,\ldots,k-1$.
	Hence there exist measurable functions $F_{n,k}$, $k=0,\ldots,n-1$, such that
	\[
	f_n(s,u)
	=
	\sum_{k=0}^{n-1}
	F_{n,k}(s)\Psi_k(u),
	\]
	and moreover
	\[
	F_{n,k}(s)=0,
	\qquad s<t_k.
	\]
	
	Consider the Volterra operator
	\[
	(K(f_n)h)(s)
	=
	\int_0^s f_n(s,u)h(u)\,du.
	\]
	Substituting the above representation yields
	\[
	(K(f_n)h)(s)
	=
	\sum_{k=0}^{n-1}
	F_{n,k}(s)
	\int_0^s \Psi_k(u)h(u)\,du .
	\]
	
	Since
	\[
	\operatorname{supp}\Psi_k\subset [0,t_k],
	\]
	we have
	\[
	\int_0^s\Psi_k(u)h(u)\,du
	=
	\int_0^{t_k}\Psi_k(u)h(u)\,du
	=:c_k(h)
	\]
	whenever $s\ge t_k$.
	On the other hand,
	\[
	F_{n,k}(s)=0
	\qquad\text{for } s<t_k.
	\]
	Therefore,
	\[
	F_{n,k}(s)
	\int_0^s\Psi_k(u)h(u)\,du
	=
	F_{n,k}(s)c_k(h),
	\]
	and consequently
	\[
	(K(f_n)h)(s)
	=
	\sum_{k=0}^{n-1}
	c_k(h)F_{n,k}(s).
	\]
	
	It follows that
	\[
	\operatorname{Ran}(K(f_n))
	\subset
	\operatorname{span}
	\{F_{n,0},\ldots,F_{n,n-1}\}.
	\]
	Hence
	\[
	\operatorname{rank}(K(f_n))
	\le n.
	\]
	Thus $K(f_n)$ is a finite-rank operator and therefore nuclear. Consequently $K(\widetilde f_n)$ is a finite-rank operator. In particular,
	it is nuclear.
\end{proof}
	\section*{Conflicts of interests}
	The authors declare no conflict of interests.

	{\bf Acknowledgement}  H. Gao was supported in part by the NSFC Grant Nos. 12571164 and the Jiangsu Provincial Scientific Research Center of Applied Mathematics under Grant No. BK20233002. S. Liu  and J. Duan are supported by the National Natural Science Foundation of China (Grant
	No. 12141107), the Guangdong Provincial Key Laboratory of Mathematical and Neural Dynamical
	Systems (Grant No. 2024B1212010004), the CrossDisciplinary ResearchTeam on Data Science and
	Intelligent Medicine (Grant No. 2023KCXTD054), and the Guangdong-Dongguan Joint Research
	(Grant No. 2023A151514 0016).

	
	\bibliographystyle{amsplain}
	
	\bibliography{refFBM}
	
\end{document}